\documentclass[11pt]{amsart}      
\usepackage{amsmath,amssymb,amsthm,color,enumerate,enumitem}
\usepackage{hyperref} 
\usepackage[normalem]{ulem} 
\usepackage{xcolor}

\newtheorem{thm}{Theorem}
\newtheorem{thmx}{Theorem}

\newtheorem{prop}[thm]{Proposition}
\newtheorem{cor}[thm]{Corollary}
\newtheorem{lemma}[thm]{Lemma}
\newtheorem{fact}[thm]{Fact}
\newtheorem{question}[thm]{Question}
\theoremstyle{definition}
\newtheorem{defin}[thm]{Definition}
\newtheorem{convention}[thm]{Convention}
 
\theoremstyle{remark}
\newtheorem{remark}[thm]{Remark}

\newcommand{\er}{\mathbb R}
\newcommand{\en}{\mathbb N}

\newcommand{\setsep}{:\;}

\newcommand{\A}{\mathcal A}
\newcommand{\F}{\mathcal F}
\newcommand{\HH}{\mathcal H}
\newcommand{\B}{\mathcal B}
\newcommand{\C}{\mathcal C}
\newcommand{\U}{\mathcal U}
\newcommand{\M}{\mathcal M}

\newcommand{\SSS}{\mathcal S}

\newcommand{\w}{\operatorname{w}}
\newcommand{\dom}{\operatorname{Dom}}
\def \rng {\operatorname{Rng}}

\def \alg{\operatorname{alg}} 
\def \suppt {\operatorname{suppt}}

\makeatletter
\@namedef{subjclassname@2020}{\textup{2020} Mathematics Subject Classification} 
\makeatother



\begin{document}

\title[Characterization of Eberlein compacta]{Characterization of (semi-)Eberlein compacta using retractional skeletons}

\author[C. Correa]{Claudia Correa}
\address[C. Correa]{Centro de Matem\'atica, Computa\c c\~ ao e Cogni\c c\~ ao, Universidade Federal do ABC, Avenida dos Estados, 5001, Santo Andr\' e, Brasil}
\email{claudiac.mat@gmail.com, claudia.correa@ufabc.edu.br}

\author[M. C\'uth]{Marek C\'uth}
\address[M. C\'uth]{Faculty of Mathematics and Physics, Department of Mathematical Analysis\\
Charles University\\
186 75 Praha 8\\
Czech Republic}
\email{cuth@karlin.mff.cuni.cz}
\address[M. C\'uth]{Institute of Mathematics of the Czech Academy of Sciences, \v{Z}itn\'a 25, 115 67 Prague 1, Czech Republic}

\author[J. Somaglia]{Jacopo Somaglia}
\address[J. Somaglia]{Dipartimento di Matematica ``F. Enriques''\\ Universit\`a degli Studi di Milano\\ Via Cesare Saldini 50\\ 20133 Milano\\ Italy}
\email{jacopo.somaglia@unimi.it}

\begin{abstract}We deeply study retractions associated to suitable models in compact spaces admitting a retractional skeleton and find several interesting consequences. Most importantly, we provide a new characterization of Valdivia compacta using the notion of retractional skeletons, which seems to be helpful when characterizing its subclasses. Further, we characterize Eberlein and semi-Eberlein compacta in terms of retractional skeletons and show that our new characterizations give an alternative proof of the fact that continuous image of an Eberlein compact is Eberlein as well as new stability results for the class of semi-Eberlein compacta, solving in particular an open problem posed by Kubis and Leiderman.
\end{abstract}

\thanks{C. Correa has been partially supported by Funda\c c\~ao de Amparo \`a Pesquisa do Estado de S\~ao Paulo (FAPESP) grants 2018/09797-2 and 2019/08515-6. M.~C\'uth has been supported by Charles University Research program No. UNCE/SCI/023 and by the GA\v{C}R project 19-05271Y. J.~Somaglia has been supported  by Universit\`a degli Studi di Milano, Research Support Plan 2019 and by Gruppo Nazionale per l'Analisi Matematica, la Probabilit\`a e le loro Applicazioni (GNAMPA) of Istituto Nazionale di Alta Matematica (INdAM), Italy.}

\maketitle

\section{Introduction}

The study of the class of compact spaces that admit a retractional skeleton was initiated in \cite{KM06}, where the authors proved that a compact space is Valdivia if and only if it admits a commutative retractional skeleton. Later, in \cite{kubis09} a notion similar to retractional skeletons in the context of Banach spaces was introduced; namely, the notion of projectional skeletons. In some sense, those notions are dual to each other. More precisely, if a compact space $K$ admits a retractional skeleton, then $\C(K)$ admits a projectional skeleton and if a Banach space $X$ admits a projectional skeleton, then $(B_{X^*},w^*)$ admits a retractional skeleton. The class of Banach (compact) spaces with a projectional (retractional) skeleton was deeply investigated from various perspectives and nowadays we have quite a rich family of natural examples and interesting results related to various fields of mathematics such as topology \cite{So20}, Banach space theory \cite{FM18}, theory of von Neumann algebras \cite{BHK16} or $JBW^*$-triples \cite{BHKPP18}. Let us note that, quite surprisingly, there was independently introduced also the notion of monotonically retractable topological spaces which turned out to be very closely related to the study of compact spaces that admit a retractional skeleton, see \cite{CK15}, and from there on, several results and modifications of the corresponding notions were considered, see e.g. \cite{CGR17, GFRH16, GFYA20}.

One of the recent streams in the area is to describe some classes of Banach (compact) spaces using the notion of projectional (retractional) skeletons, see e.g. \cite{KM06, CF17, FM18, kubis09} where the characterizations of Plichko spaces (and Valdivia compacta), WLD spaces (and Corson compacta), Asplund spaces, WLD+Asplund spaces and WCG spaces were given.

The main two results of this paper (Theorems \ref{thm:Intro1} and \ref{thm:Intro2}) are characterizations of Eberlein and semi-Eberlein compacta, respectively, using the notion of retractional skeletons. 
Let us recall that given a set $I$ we define
\[
c_0(I):=\{x\in \er^I\colon (\forall \varepsilon > 0) |\{i\in I\colon |x(i)|>\varepsilon\}| < \omega\}\subset \er^I
\]
and that a compact space $K$ is \emph{Eberlein} if it homeomophically embeds into $c_0(I)$, for some set $I$. This is a central concept in Banach space theory, as it is known that a compact space is Eberlein if and only if it is homeomorphic to a weakly compact set of a Banach space, see \cite{AL68} or \cite[Corollary 13.19]{FHHMZ}. For the notion of shrinkingness we refer the reader to Definition~\ref{def:shrinking}.

\begin{thmx}\label{thm:Intro1}
Let $K$ be a compact space. Then the following conditions are equivalent:
\begin{enumerate}
    \item $K$ is Eberlein.
    \item There exist a bounded set $\A\subset \C(K)$ separating the points of $K$ and a retractional skeleton $\mathfrak{s} = (r_s)_{s\in\Gamma}$ on $K$ such that $\mathfrak{s}$ is $\A$-shrinking.
    \item There exist a countable family $\A$ of subsets of $B_{\C(K)}$ and a full retractional skeleton $\mathfrak{s} = (r_s)_{s\in\Gamma}$ on $K$ such that
\begin{enumerate}[label=(\alph*)]
    \item For every $A\in\A$ there exists $\varepsilon_A>0$ such that $\mathfrak{s}$ is $(A,\varepsilon_A)$-shrinking, and
    \item\label{cond:fullInBall0} for every $\varepsilon>0$ we have $B_{\C(K)} = \bigcup\{A\in \A\colon \varepsilon_A < \varepsilon\}$.
\end{enumerate}
\end{enumerate}
\end{thmx}

Recall that a compact space $K$ is Eberlein if and only if $\C(K)$ is WCG if and only if $\C(K)$ is a subspace of a WCG space, thus Theorem \ref{thm:Intro1} is naturally connected to the characterization of WCG Banach spaces and their subspaces presented in \cite{FM18}. Moreover, from Theorem~\ref{thm:Intro1} one may deduce that continuous images of Eberlein compacta are Eberlein, see Remark \ref{rem:contImageEberlein} below. Quite many steps of our proof seem to be much more flexible and we believe that those may be used in order to find characterizations of other natural subclasses of Valdivia compacta (the most important in this respect is probably Theorem~\ref{thm:Intro3} mentioned below). This is witnessed by the characterization of semi-Eberlein compacta presented in Theorem \ref{thm:Intro2}. Recall that, following \cite{KL04}, we say a compact space $K$ is \emph{semi-Eberlein} if there exists a homeomorphic embedding $h:K\to\er^I$ such that $h^{-1}[c_0(I)]$ is dense in $K$. We denote by $D(\mathfrak s)$ the set induced by a retractional skeleton $\mathfrak s$ (see Definition \ref{def: r-skeleton}).

\begin{thmx}\label{thm:Intro2}
Let $K$ be a compact space. Then the following conditions are equivalent:
\begin{enumerate}
    \item\label{it: semi-Eberlein} $K$ is semi-Eberlein.
    \item\label{it: semi-EberleinRetractionalSkeleton} There exist a dense subset $D\subset K$, a bounded set $\A\subset \C(K)$ separating the points of $K$ and a retractional skeleton $\mathfrak{s} = (r_s)_{s\in\Gamma}$ on $K$ with $D\subset D(\mathfrak{s})$ such that
\begin{enumerate}[label=(\alph*)]
    \item $\mathfrak{s}$ is $\A$-shrinking with respect to $D$, and
        \item $\lim_{s\in\Gamma'} r_s(x)\in D$, for every $x\in D$ and every up-directed subset $\Gamma'$ of $\Gamma$.
\end{enumerate}
\item  There exist a dense set $D\subset K$, a countable family $\A$ of subsets of $B_{\C(K)}$ and a retractional skeleton $\mathfrak{s} = (r_s)_{s\in\Gamma}$ on $K$ with $D\subset D(\mathfrak{s})$ such that
\begin{enumerate}[label=(\alph*)]
    \item For every $A\in\A$ there exists $\varepsilon_A>0$ such that $\mathfrak{s}$ is $(A,\varepsilon_A)$-shrinking with respect to $D$,
    \item\label{cond:fullInBall2} for every $\varepsilon>0$ we have $B_{\C(K)} = \bigcup\{A\in \A\colon \varepsilon_A < \varepsilon\}$, and
    \item\label{it:inSpecialCaseEquivalentToCommutativity0} $\lim_{s\in\Gamma'} r_s(x)\in D$, for every $x\in D$ and every up-directed subset $\Gamma'$ of $\Gamma$.
\end{enumerate}
\end{enumerate}
\end{thmx}

Finally, using Theorem~\ref{thm:Intro2} we provide new structural results for the class of semi-Eberlein compacta, answering in particular the second part of \cite[Question 6.6]{KL04} in positive. The most important new stability results are summarized below.
\begin{thmx}\label{thm:Intro4}
Let $K$ be a semi-Eberlein compact space.
\begin{itemize}
    \item If $L$ is an open continuous image of $K$ and it has densely many $G_\delta$-points, then $L$ is semi-Eberlein.
    \item If $K$ is moreover Corson and $L$ is a continuous image of $K$, then $L$ is semi-Eberlein.
\end{itemize}
\end{thmx}

As mentioned previously, many steps of the proofs of Theorem~\ref{thm:Intro1} and Theorem~\ref{thm:Intro2} are of independent interest and we believe those could be used when trying to characterize other subclasses of Valdivia compacta, which opens quite a wide area of potential further research. This is outlined in Section~\ref{sec:questions}.

Let us now briefly describe the content of each section, emphasizing the general steps mentioned above.

Section~\ref{sec:prelim} contains basic notations and some preliminary results.

In Section~\ref{sec:models} we consider retractions associated to (not necessary countable) suitable models. The most important outcome is Theorem~\ref{thm:canonicalRetraction}, where we summarize the properties of canonical retractions associated to suitable models. As an easy consequence, in Proposition~\ref{prop:rri} we show a very general method of obtaining a continuous chain of retractions on a compact space admitting a retractional skeleton. This part is essentially known as similar results were obtained e.g. in \cite[Lemma 2.5]{CGR17} (using other methods than suitable models), but our approach is in a certain sense much more flexible (most importantly, because it may be combined with other statements involving suitable models) and we actually use this flexibility later. As a corollary of our investigations we show in Theorem~\ref{thm:subskeletons} that we may in a certain way combine properties of countably many retractional skeletons.

In Section~\ref{sec:valdivia}, inspired by the proof of \cite[Theorem 2.6]{CGR17}, we aim at seeing as concretely as possible the ``Valdivia embedding'' of compact spaces with a commutative retractional skeleton. As a consequence we obtain the following result which might be thought of as the fourth main result of the whole paper. The most important part which we use later is the implication \ref{it:inducedByCommutative}$\Rightarrow$\ref{it:sigmaEmbedding}.

\begin{thmx}\label{thm:Intro3}
    Let $K$ be a compact space and $\mathfrak{s}=(r_s)_{s\in\Gamma}$ be a retractional skeleton on $K$. 
    Then the following conditions are equivalent.
    \begin{enumerate}[label = (\roman*)]  
     \newcounter{saveenum2} 

     \item\label{it:inducedByCommutative}$D(\mathfrak{s})$ is induced by a commutative retractional skeleton. 
    \item\label{it:subsetOfComutativeInducedSubset} There exists a subskeleton of $\mathfrak{s}$ which is commutative.
    
    \item\label{it:skeletonInvariantSubset} There exist a subskeleton $\mathfrak{s}_2 = (r_s)_{ s\in\Gamma'}$ of $\mathfrak{s}$ and a dense set $D\subset D(\mathfrak{s})$ such that for every up-directed set $\Gamma''\subset \Gamma'$ and every $x\in D$ we have $\lim_{s\in\Gamma''}r_s(x)\in D$.
    
    \setcounter{saveenum2}{\value{enumi}} 
    \end{enumerate}
 Moreover, if $\lambda \geq 1$ and $\A\subset \lambda B_{\C(K)}$ is a closed, symmetric and convex set separating the points of $K$ such that $f\circ r_s\in\A$, for every $f\in\A$ and $s\in\Gamma$, then those conditions are also equivalent to the following one.
 \begin{enumerate}[label = (\roman*)]
\setcounter{enumi}{\value{saveenum2}} 
 \item\label{it:sigmaEmbedding} There exists $\HH\subset \A$ such that the mapping $\varphi:K\to [-1,1]^\HH$ defined as $\varphi(x)(h):=\tfrac{h}{\lambda}(x)$, for every $h\in\HH$ and $x \in K$, is a homeomorphic embedding and $\varphi[D(\mathfrak{s})]\subset \Sigma(\HH)$.
 \end{enumerate}
\end{thmx}

Note that Theorem \ref{thm:Intro3} provides a characterization of Valdivia compacta, since a compact space is Valdivia if and only if it admits a commutative retractional skeleton. 

In Section~\ref{sec:semiEberlein} we prove (slightly more general versions of) Theorem~\ref{thm:Intro1} and Theorem~\ref{thm:Intro2}. Section~\ref{sec:contImages} is devoted to applications (in particular to the proof of Theorem~\ref{thm:Intro4}) and  Section~\ref{sec:questions} is devoted to open problems and remarks.

\section{Notation and preliminary results}\label{sec:prelim}

We use standard notations from topology and Banach space theory as can be found in \cite{Eng} and \cite{FHHMZ}.

For a set $I$, we define
\[\Sigma(I):=\{x\in \er^{I}\colon |\suppt(x)|\leq \omega\},
\]
where $\suppt(x)=\{i\in I\colon x(i)\neq 0\}$ is the \emph{support} of $x$. Given a subset $S$ of $I$ we denote the characteristic   function of $S$ by $1_S$.

All topological spaces are assumed to be Tychonoff. Let $T$ be a topological space.  A subset $S\subset T$ is said to be \emph{countably closed} if $\overline{C}\subset S$, for every countable subset $C\subset S$. We denote by $\w(T)$ the weight of $T$, by $\C(T,T)$ the set of continuous functions from $T$ to $T$ and by  $\beta T$ the \v{C}ech-Stone compactification of $T$. If $T$ is compact, then as usual $\C(T)$ denotes the Banach algebra of real-valued continuous functions defined on $T$, endowed with the supremum norm. Moreover, if $\A\subset \C(T)$, we denote by $\alg(\A)$ the algebraic hull of $\A$ in the algebra $\C(T)$.
Recall that a compact space $T$  is said to be \emph{Valdivia} if there is a homeomorphic embedding $h:T\to\er^I$ such that $h^{-1}[\Sigma(I)]$ is dense in $T$, we refer to \cite{K00} for a survey in this subject.

Let $(\Gamma,\leq)$ be an up-directed partially ordered set. We say that a sequence $(s_n)_{n\in\omega}$ of elements of $\Gamma$ is increasing if $s_n\leq s_{n+1}$, for every $n\in\omega$. We say that  $\Gamma$ is \emph{$\sigma$-complete} if for every increasing sequence $(s_n)_{n\in\omega}$ in $\Gamma$ there exists $\sup_n s_n$ in $\Gamma$. We say that $\Gamma'\subset\Gamma$ is \emph{cofinal} in $\Gamma$ if for every $s_0\in \Gamma$ there is $s\in\Gamma'$ with $s\geq s_0$. If $\Gamma$ is $\sigma$-complete and $A\subset \Gamma$, we denote by $A_{\sigma}$ the smallest $\sigma$-closed subset of $\Gamma$ containing $A$. Notice that, by \cite[Proposition 2.3]{K20}, if $A$ is up-directed, then $A_{\sigma}$ is up-directed.

\begin{defin}\label{def: r-skeleton}
Following \cite{CK15}, a \emph{retractional skeleton} in a countably compact space $K$ is a family of continuous retractions $\mathfrak s=(r_{s})_{s\in\Gamma}$ on $K$ indexed by an up-directed, $\sigma$-complete partially ordered set $\Gamma$, such that:
\begin{enumerate}[label = (\roman*)]
\item $r_{s}[K]$ is a metrizable compact space for each $s\in\Gamma$,
\item $s,t\in\Gamma$, $s\leq t$ then $r_{s}=r_t\circ r_s=r_s\circ r_t$,
\item given an increasing sequence $(s_n)_{n\in\omega}$ in $\Gamma$, if $s=\sup_{n\in\omega}s_{n}\in\Gamma$, then $r_{s}(x)=\lim_{n\to \infty}r_{s_{n}}(x)$, for every $x\in K$,
\item for every $x\in K$, $x=\lim_{s\in\Gamma}r_{s}(x)$.
\end{enumerate}
We say that $\bigcup_{s\in\Gamma}r_{s}[K]$ is the set \emph{induced} by the retractional skeleton $\mathfrak s$ and we denote it by $D(\mathfrak{s})$. We say that $\mathfrak{s}$ is \emph{commutative} if we have $r_s\circ r_t = r_t \circ r_s$ for every $s,t\in\Gamma$. We say that $\mathfrak{s}$ is \emph{full} if $D(\mathfrak{s}) = K$.
\end{defin}

The following preliminary result will be used in what follows quite frequently. It seems to be new even though it could be known to some experts as well.

\begin{lemma}\label{l: retraction from up-directed subset}
Let $K$ be a compact space. Suppose that $K$ has a retractional skeleton $\mathfrak{s}=(r_s)_{s\in \Gamma}$. Let $\Gamma^{'}\subset \Gamma$ be an up-directed subset, then the mapping $R_{\Gamma^{'}}:K \to K$ defined by $R_{\Gamma^{'}}(x)=\lim_{s\in \Gamma^{'}}r_{s}(x)$ is a continuous retraction and $R_{\Gamma^{'}}[K] = \overline{\bigcup_{s\in\Gamma'}r_s[K]}$.
Moreover, the following holds.
\begin{enumerate}[label = (\roman*)]
\item\label{it:countableUpDirected} If $\Gamma'$ is countable, then $s=\sup\Gamma'$ exists and we have $R_{\Gamma'} = r_s$.
\item\label{it:generalUpDirected} If $\M$ is an up-directed subset of $\mathcal{P}(\Gamma)$ such that each $M\in\M$ is up-directed. Then $\lim_{M\in\M}R_{M}(x) = R_{\bigcup \M}(x)$, $x\in K$.
\item\label{it:sigmaClosure} For every $s\in (\Gamma')_\sigma$ we have that $r_s[K]\subset R_{\Gamma'}[K]$ and $r_s\circ R_{\Gamma'} = R_{\Gamma'}\circ r_s$.
\item\label{it:skeletonOnTheSubspace} $(r_s|_{R_{\Gamma'}[K]})_{s\in(\Gamma')_\sigma}$ is a retractional skeleton on $R_{\Gamma'}[K]$ with induced set $D(\mathfrak{s})\cap R_{\Gamma'}[K]$.
\item\label{it:commutativeCase} If $\mathfrak{s}$ is commutative, then $D(\mathfrak{s})\cap R_{\Gamma'}[K] = R_{\Gamma'}[D(\mathfrak{s})]$.
\end{enumerate}
\end{lemma}
\begin{proof}Let us start by proving that the mapping $R_{\Gamma'}$ is well-defined. In order to do that fix $x\in K$ and suppose that $(r_{s}(x))_{s\in \Gamma^{'}}$ is an infinite set (otherwise the assertion would be trivial). Since $K$ is compact, there exists a cluster point $x_1\in K$ for the net $(r_{s}(x))_{s\in \Gamma^{'}}$. Let us show that such a cluster point $x_1$ is unique. Indeed, let $x_{1}\neq x_{2}$ be two cluster points of $(r_{s}(x))_{s\in \Gamma^{'}}$. Let  $U_{1},U_{2}\subset K$ be two open subsets such that $x_{1}\in U_{1}$, $x_{2}\in U_{2}$ and $\overline{U_{1}}\cap \overline{U_{2}}=\emptyset$. Let $(s_n)_{n<\omega}, (t_{n})_{n<\omega}\subset \Gamma'$ be two increasing sequences of indexes
such that $s_n\leq t_n\leq s_{n+1}$, $r_{s_{n}}(x)\in U_{1},$ and $r_{t_{n}}(x)\in U_{2}$, for every $n\in \omega$. Since $\Gamma$ is $\sigma$-complete, we have that $\sup_{n\in\omega}s_{n}=\sup_{n\in\omega}t_{n}=s\in \Gamma$. Then $r_{s}(x)\in \overline{U_{1}}\cap \overline{U_{2}}$, a contradiction. Therefore $R_{\Gamma'}$ is well-defined.\\
The map $R_{\Gamma^{'}}$ is continuous. Indeed, let $(x_{\lambda})_{\lambda\in\Lambda}$ be a net converging to $x\in K$. Up to taking a subnet we may assume without loss of generality that $R_{\Gamma^{'}}(x_{\lambda})$ converges to $y$. Suppose by contradiction $y\neq R_{\Gamma^{'}}(x)$, then there are two open subsets $U,V\subset K$ with $y\in U$ and $ R_{\Gamma^{'}}(x)\in V$, such that $\overline{U}\cap\overline{V}=\emptyset$. We find recursively two increasing sequences of indexes $(s_n)_{n<\omega}$ in $\Gamma^{'}$ and $(\lambda_n)_{n<\omega}$ in $\Lambda$ such that $r_{s_k}(x_{\lambda_i})\in V \mbox{ if } i\geq k$ and $r_{s_k}(x_{\lambda_i})\in U \mbox{ if } i< k$.\\
Let us sketch the recursion here. Since $R_{\Gamma^{'}}(x_{\lambda})\to y$, there exists $\lambda_0\in\Lambda$ such that $R_{\Gamma^{'}}(x_{\lambda})\in U$ for every $\lambda\geq \lambda_0$.  Since $r_s(x)\stackrel{\Gamma'}{\to}R_{\Gamma' (x)}$, there exists $s_0\in \Gamma^{'}$ such that $r_{t}(x)\in V$ for every $t\geq s_0$. Since $r_s(x_{\lambda_0})\stackrel{\Gamma'}{\to}R_{\Gamma^{'}}(x_{\lambda_0})\in U$, there exists $s_1\geq s_0$ such that $r_t(x_{\lambda_0})\in U$ for every $t\geq s_1$. By the continuity of $r_{s_1}$, we have $r_{s_1}(x_{\lambda})\to r_{s_1}(x)\in V$; hence there exists $\lambda_1\geq\lambda_0$ such that $r_{s_1}(x_{\lambda})\in V$ for every $\lambda\geq\lambda_1$. We proceed recursively in an obvious way.\\
Since $\Gamma$ is $\sigma$-complete, $s=\sup_{k\in \omega}s_k$ belongs to $\Gamma$. Hence $r_{s_k}(x_{\lambda_i})$ converges to $r_{s}(x_{\lambda_i})\in\overline{U}$ for every $i\in \omega$. Moreover, by compactness we have $\bigcap_{k\in\omega}\overline{(x_{\lambda_i})_{i\geq k}}\neq \emptyset$, so we may pick $\tilde{x}\in\bigcap_{k\in\omega}\overline{(x_{\lambda_i})_{i\geq k}}$. 
We observe that $r_{s_k}(\tilde{x})\in \overline{V}$ for every $k\in\omega$, hence $r_s(\tilde{x})\in \overline{V}$. On the other hand $r_{s_k}(x_{\lambda_i})\to r_s(x_{\lambda_i})\in \overline{U}$ for every $i\in\omega$; therefore $r_s(\tilde{x})\in\overline{U}$, a contradiction. Thus, $R_{\Gamma'}$ is continuous.\\ 
Let us check that $R_{\Gamma'}$ is a retraction. Indeed, pick $x\in K$. Then
\begin{equation*}
\begin{split}
R_{\Gamma^{'}}(R_{\Gamma^{'}}(x))&=\lim_{t\in\Gamma^{'}}r_t(\lim_{s\in\Gamma^{'}}r_s(x))
=\lim_{t\in\Gamma^{'}}\lim_{s\in\Gamma^{'}}r_t(r_s(x))\\
&=\lim_{t\in\Gamma'}\lim_{s\in\Gamma', s\geq t}r_t(r_s(x))
=\lim_{t\in\Gamma^{'}}r_t(x)=R_{\Gamma^{'}}(x).
\end{split}
\end{equation*}
Finally, for every $s\in\Gamma'$ and $x\in r_s[K]$ we have $R_{\Gamma'}(x)=\lim_{t\in\Gamma', t\geq s} r_t(r_{s}(x))=x$ so we obtain $\overline{\bigcup_{s\in\Gamma'}r_{s}[K]}\subset R_{\Gamma'}[K]$ and the other inclusion follows from the definition of $R_{\Gamma'}$.\\
It remains to prove the ``Moreover'' part. We first observe (see the proof of \cite[Proposition 2.3]{K20} for more details) that $(\Gamma')_{\sigma}$ is directed, $\sigma$-closed and $(\Gamma')_{\sigma}=\bigcup_{\alpha<\omega_1}B_{\alpha}$, where
\begin{itemize}
    \item $B_0=\Gamma'$;
    \item $B_{\alpha+1}=B_{\alpha}\cup \{\sup t_n : \; (t_n) \mbox{ is an increasing sequence in $B_{\alpha}$}\};$
    \item $B_{\lambda}=\bigcup_{\alpha<\lambda}B_{\alpha}$, if $\lambda<\omega_1$ is a limit ordinal.
\end{itemize}
\noindent\ref{it:countableUpDirected}: If $\Gamma'$ is countable, then we can find an increasing sequence $(s_n)_{n\in\omega}$ from $\Gamma$ with $\sup_n s_n = s = \sup \Gamma'$. Then, using that the sequence $\{s_n\colon n\in\omega\}$ is cofinal in $\Gamma'$, we obtain $R_{\Gamma'} = R_{\{s_n\colon n\in\omega\}} = r_s$.\\
\ref{it:generalUpDirected}: Suppose that $\M\subset \mathcal{P}(\Gamma)$ is up-directed and that each $M\in\M$ is up-directed. Put $M_\infty:=\bigcup_{M\in \M} M$, fix $x\in K$ and open set $U$ such that $R_{M_{\infty}}(x) \in U$. Let $V$ be an open neighborhood of $R_{M_{\infty}}(x)$ such that $\overline{V} \subset U$. Then, there exists $s_0 \in M_{\infty}$ such that $r_s(x) \in V$, for every $s\in M_\infty$ with $s\geq s_0$.
By the definition of $M_{\infty}$, there exists $M_0\in \M$ such that $s_0 \in M_0$. If $M \in \M$ and  $M_0 \subset M$, then $s_0 \in M$. This implies that the set $\{s\in M\colon s\geq s_0\}$ is cofinal in $M$ and so we have 
\[R_M(x)=\lim_{s \in M}r_s(x)= \lim_{s \in M, s \ge s_0}r_s(x)\in \overline{V} \subset U.\]
This shows that $\lim_{M\in\M} R_M(x) = R_{M_\infty}(x)$.\\
\ref{it:sigmaClosure}:  We prove inductively that for every $\alpha<\omega_1$ and $s\in B_\alpha$, it holds that $r_s[K]\subset R_{\Gamma'}[K]$ and $r_s\circ R_{\Gamma'} = R_{\Gamma'}\circ r_s$. Pick $s\in B_0 = \Gamma'$. Then $r_s[K]\subset R_{\Gamma'}[K]$, since $R_{\Gamma^{'}}[K] = \overline{\bigcup_{s\in\Gamma'}r_s[K]}$. Moreover, for $x\in K$ we have 
\[
r_s\big(R_{\Gamma'}(x)\big)=\lim_{t \in \Gamma', t \ge s} r_s\big(r_t(x)\big)=\lim_{t \in \Gamma', t \ge s} r_t\big(r_s(x)\big)=R_{\Gamma'}\big(r_s(x)\big).\]
Now, fix $\alpha<\omega_1$ and suppose that the result holds for every $\gamma<\alpha$. If $\alpha$ is a limit ordinal, then it follows easily from the induction hypothesis that the result also holds for $\alpha$. Suppose that $\alpha=\gamma+1$. Let $s\in B_{\alpha}$, $x\in r_{s}[K]$ and $(s_n)_{n\in\omega}\subset B_{\gamma}$ such that $\sup s_n=s$. By the induction hypothesis, we have that $R_{\Gamma'}(r_{s_n}(x))=r_{s_n}(x)$, for every $n \in \omega$ and therefore:
\[R_{\Gamma'}(x)=\lim_{n\in\omega}R_{\Gamma'}(r_{s_n}( x))=\lim_{n\in\omega}r_{s_n}(x)=r_s(x)=x.\]
With a similar argument, we also conclude that $r_s\circ R_{\Gamma'} = R_{\Gamma'}\circ r_s$.\\
\ref{it:skeletonOnTheSubspace}: First, we \emph{claim} that for every $x\in R_{\Gamma'}[K]$ we have $\lim_{s\in (\Gamma')_\sigma} r_s(x) = x$. Indeed, since $(\Gamma')_{\sigma}$ is up-directed, it holds that $R_{(\Gamma')_{\sigma}}[K]=\overline{\bigcup_{s \in (\Gamma')_{\sigma}}r_s[K]}$, which implies that $R_{\Gamma'}[K] \subset R_{(\Gamma')_{\sigma}}[K]$ and therefore if $x \in  R_{\Gamma'}[K]$, then $x=R_{(\Gamma')_\sigma}(x)=\lim_{s\in (\Gamma')_\sigma} r_s(x)$.

Using \ref{it:sigmaClosure} and the previous claim, it is easy to see that $\mathfrak{s'}:=(r_s|_{R_{\Gamma'}[K]})_ {s\in(\Gamma')_\sigma}$ is a retractional skeleton on $R_{\Gamma'}[K]$ with $D(\mathfrak{s'}) = \bigcup_{s\in(\Gamma')_\sigma} r_s[R_{\Gamma'}[K]]\subset D(\mathfrak{s})\cap R_{\Gamma'}[K]$. On the other hand, since $D(\mathfrak{s})$ is Fr\'echet-Urysohn (see \cite[Theorem 32]{kubis09}), for every $x\in D(\mathfrak{s})\cap R_{\Gamma'}[K]$ there is a sequence $(s_n)_{n\in\omega}$ in $\Gamma'$ with $r_{s_n}(x)\to x$ and therefore $x\in D(\mathfrak{s'})$, because $D(\mathfrak{s'})$ is a countably closed set. Thus, we have that $D(\mathfrak{s'}) = D(\mathfrak{s})\cap R_{\Gamma'}[K]$.\\
\ref{it:commutativeCase}: If $(r_s)_{s\in\Gamma}$ is commutative, then for every $s\in\Gamma$ and $x\in K$ we have
\[R_{\Gamma'}(r_s(x)) = \lim_{t\in\Gamma'} r_t(r_s(x)) = r_s(\lim_{t\in\Gamma'} r_t(x))\in D(\mathfrak{s}),\]
which implies $R_{\Gamma'}[D(\mathfrak{s})]\subset D(\mathfrak{s})$ and so $R_{\Gamma'}[D(\mathfrak{s})] = D(\mathfrak{s})\cap R_{\Gamma'}[K]$.
\end{proof}

\section{Retractions associated to suitable models}\label{sec:models}

The most important results concerning projectional skeletons were originally proved in \cite{kubis09} using the so-called ``method of suitable countable models'' which replaces inductive constructions by ``suitable countable models''. The presentation of this method was further simplified in \cite{C12} and later it was also used in the context of spaces admitting retractional skeletons, see e.g. \cite{C14} or \cite{CK15}. Here we further generalize and deeply investigate this method. The main difference of our approach is that we do not consider only countable models. The main outcome of this section is that for every (not necessarily countable) suitable model we can define a canonical retraction associated to this model. Those canonical retractions will be deeply used in the remainder of the paper.

Properties of retractions associated to suitable models are summarized in Theorem~\ref{thm:canonicalRetraction} and, consequently, in Proposition~\ref{prop:rri} we obtain a continuous chain of retractions associated to suitable models with very pleasant properties. As an example of an application we show in Theorem~\ref{thm:subskeletons} that we may in a certain way combine properties of countably many retractional skeletons.

\subsection{Preliminaries}

Here we settle the notation and give some basic observations concerning suitable models. We refer the interested reader to \cite{C12} and \cite{CK15}, where more details about this method may be found (warning: in \cite{C12, CK15} only \textbf{countable} models were considered, while here we consider suitable models which are not necessarily countable).

Any formula in the set theory can be written using symbols $\in,=,\wedge,\vee,\neg,\rightarrow,\leftrightarrow,\exists,(,),[,]$ and symbols for variables. On the other hand, it would be very laborious and pointless to use only the basic language of the set theory. For example, we often write $x < y$ and we know, that in fact this is a shortcut for a formula $\varphi(x,y,<)$ with all free variables shown. Thus, in what follows we will use this extended language of the set theory as we are used to, having in mind that the formulas we work with are actually sequences of symbols from the list mentioned above.

Let $N$ be a fixed set and $\phi$ be a formula. Then the {\em relativization of $\phi$ to $N$} is the formula $\phi^N$ which is obtained from $\phi$ by replacing each quantifier of the form ``$\exists x$'' by ``$\exists x\in N$'' (and if we extend our language of  set theory by the symbol ``$\forall$'' then we replace also each quantifier of the form ``$\forall x$'' by ``$\forall x\in N$'').

If $\phi(x_1,\ldots,x_n)$ is a formula with all free variables shown, then {\em $\phi$ is absolute for $N$} if
\[
\forall a_1,\ldots,a_n\in N\quad (\phi^N(a_1,\ldots,a_n) \leftrightarrow \phi(a_1,\ldots,a_n)).
\]

\begin{defin}
\rm Let $\Phi$ be a finite list of formulas and $X$ be any set.
Let $M \supset X$ be a set such that each $\phi$ from $\Phi$ is absolute for $M$.
Then we say that $M$ \emph{is a suitable model for $\Phi$ containing $X$}.
This is denoted by $M \prec (\Phi; X)$.
\end{defin}

Note that suitable models do exist.

  \begin{thm}[see Theorem IV.7.8 in \cite{kunen}]\label{thm:modelExists}
 Let $\Phi$ be a finite list of formulas and $X$ be any set. Then there exists a set $R$ such that $R\prec (\Phi; X)$ and $|R|\leq \max(\omega,|X|))$ and moreover, for every countable set $Z\subset R$ there exists $M\subset R$ such that $M\prec(\Phi;\; Z)$ and $M$ is countable.
\end{thm}

The fact that certain formula is absolute for $M$ will always be used exclusively in order to satisfy the assumption of the following lemma. Using this lemma we can force the model $M$ to contain all the needed objects created (uniquely) from elements of $M$. We give here the well-known proof for the convenience of the reader.

\begin{lemma}\label{l:unique-M}
Let $\phi(y,x_1,\ldots,x_n)$ be a formula with all free variables shown and let $M$ be a set that is absolute for $\phi$ and for $\exists y \phi(y,x_1, \ldots, x_n)$. If $a_1, \ldots, a_n\in M$ are such that there exists a set $u$ satisfying $\phi(u, a_1, \ldots, a_n)$, then there exists a set $v \in M$ satisfying $\phi(v, a_1, \ldots, a_n)$. Moreover, if there exists a unique set $u$ such that $\phi(u, a_1, \ldots, a_n)$, then $u \in M$.
\end{lemma}
\begin{proof}
It follows from the absoluteness of the formula $\exists y \phi(y,x_1, \ldots, x_n)$, that there exists $v \in M$ such that $\phi^M(v, a_1, \ldots, a_n)$. Therefore the absoluteness of the formula $\phi(y,x_1, \ldots, x_n)$ implies that $\phi(v, a_1, \ldots, a_n)$ holds. Moreover, if $u$ is the only set such that $\phi(u, a_1, \ldots, a_n)$, then $v=u$ and thus $u \in M$.
\end{proof}

\begin{convention}
Whenever we say ``\emph{for any suitable model $M$ (the following holds \dots)}''
we mean that  ``\emph{there exists a finite list of formulas $\Phi$ and a countable set $Y$ such that for every $M \prec (\Phi;Y)$ (the following holds \dots)}''.
\end{convention}

If $M$ is a suitable model and $\langle X,\tau\rangle$ is a topological space (or is $\langle X,d\rangle$ a metric space or is $\langle X,+,\cdot,\|\cdot\|\rangle$ a normed linear space) then we say that \emph{$M$ contains $X$} if $\langle X,\tau\rangle\in M$, $\langle X,d\rangle\in M$ and $\langle X,+,\cdot,\|\cdot\|\rangle\in M$, respectively.

The following summarizes certain easy observations. For the proofs we refer the reader to \cite[Sections 2 and 3]{C12}, where it is assumed that $M$ is countable but this fact is not used in proofs. 

\begin{lemma}\label{l:basics-in-M}
For any suitable model $M$ the following holds:
\begin{enumerate}
    \item \label{number-sets} $\mathbb{Q}, \omega, \er \in M$ and $M$ contains the usual operations and relations on $\er$.
    \item\label{it: Dom-Rng-in-M} For every function $f\in M$ we have $\dom f\in M$, $\rng f\in M$ and $f[M\cap \dom f]\subset M$.
    \item \label{finite-in-M} For every finite set $A$ we have $A\in M$ if and only if $A\subset M$.
    \item \label{belongIsContained}For every countable set $A\in M$ we have $A\subset M$. Moreover, if $\kappa\in M$ is a cardinal and $\kappa\subset M$ then for every $A\in M$ with $|A|\leq \kappa$ we have $A\subset M$.
    \item For every natural number $n>0$ and sets $a_1,\ldots,a_n$ we have $\{a_1,\ldots,a_n\}\subset M$ if and only if $\langle a_1,\ldots,a_n\rangle \in M$.
    \item \label{operations-sets-in-M} If $A,B\in M$, then $A\cap B\in M$, $B\setminus A\in M$ and $A\cup B\in M$.
    \item \label{linearSubspaceM}If $M$ contains a normed linear space $X$, then $\overline{X\cap M}$ is a linear subspace of $X$.
\end{enumerate}
\end{lemma}

Some more easy observations are summarized in the following. 

\begin{lemma}\label{l:basics-in-M-2}
For any suitable model $M$ the following holds:
\begin{enumerate}
    \item\label{it:upDirected} If $(\Gamma,\leq)$ is up-directed and $(\Gamma,\leq)\in M$, then $\Gamma\cap M$ is up-directed.
    \item\label{it:composition} If $f,g\in M$ are functions and $f\circ g$ is well-defined, then $f\circ g\in M$.
    \item \label{it:inverse-in-M} If $f \in M$ is a function which is one-to-one then $f^{-1}\in M$.
    \item\label{it:one-to-one-imageOfM} If $f\in M$ is a function and $X\in M$ is a subset of $\dom f$, then $f[M\cap X] = M\cap f[X]$.
    \item\label{it:exp-in-M} If $A$ and $B$ are sets and $A, B \in M$, then $B^A \in M$ and $A\times B \in M$.
    \item \label{it:Pi-in-M} For every set $I\in M$ and $X\subset \er^I$ with $X\in M$ we have $\pi\in M$, where $\pi:I\to {\er^X}$ is the mapping given for $i\in I$ and $x\in X$ as $\pi(i)(x):=x(i)$.
    \item\label{it:supportSubsetOfM} Let $X\subset \Sigma(I)$ be such that $I\in M$. Then $\suppt(x)\subset M$ for every $x\in X\cap M$.
    \item \label{it:containsspacecontinuousfunctions}If $(X,\tau)$ is a topological space with $\{X,\tau\}\subset M$, then $\{\C(X),+,\cdot,\otimes\}\subset M$ (where $\cdot$ is a multiplication by real numbers and $\otimes$ pointwise multiplication of functions). Moreover, if $X$ is a compact space then $M$ contains the normed linear space $\C(X)$, $\overline{\C(X)\cap M}$ is a closed subalgebra of $\C(X)$ and $1\in\C(X)\cap M$.  
    \item\label{it:algebraDenseInFunctionsFromM} If $(K,\tau)$ is a compact space, $\A\subset \C(K)$ separates the points of $K$ and  $\{K,\tau,\A\}\subset M$, then $\overline{\alg((\A\cup \{1\})\cap M)} = \overline{\C(K)\cap M}$.
    \item\label{it:metrizableSubspace} If $(K,\tau)$ is a compact space , $K'\subset K$ is closed and metrizable with $\{K',\tau,K\}\in M$ then $\C(K)\cap M$ separates the points of $K'$ and $K'\subset \overline{K'\cap M}$.
    \item\label{it:norm} If $(K,\tau)$ is a compact space, $D\subset K$ a dense subset with $\{K,D,\tau\}\subset M$ and $f\in\C(K)\cap M$, then $\|f\| = \|f|_{D\cap M}\|$
\end{enumerate}
\end{lemma}
\begin{proof}
Let $S$ and $\Phi$ be the countable set and the list of formulas from the statement of Lemma~\ref{l:basics-in-M}, where $\Phi$ is enriched by formulas (and their subformulas) marked by $(*)$ in the proof below. Let $M\prec (\Phi; S)$. Then $M$ satisfies \eqref{it:upDirected}, \eqref{it:composition}, \eqref{it:inverse-in-M}, \eqref{it:exp-in-M}, and \eqref{it:Pi-in-M}. Indeed those items follow easily using Lemma~\ref{l:unique-M} and the absoluteness of the following formulas (and their subformulas)
\[
\forall u,v \in \Gamma\,\exists w\in\Gamma\, w\geq u,v,\eqno{(*)}
\]
\[
\exists h \quad (h= f\circ g).\eqno{(*)}
\]
\[
\exists h \quad (h=f^{-1}).\eqno{(*)}
\]
\[
\exists W \quad (W = B^A).\eqno{(*)}
\]
\[
\exists W \quad (W = B\times A).\eqno{(*)}
\]
\[
\exists \pi\in ({\er^X})^{I} \quad (\forall i\in I \ \forall x\in X: \pi(i)(x) = x(i)).\eqno{(*)}
\]

\eqref{it:one-to-one-imageOfM}: By Lemma~\ref{l:basics-in-M} \eqref{it: Dom-Rng-in-M}, we have that $f[M\cap \dom f]\subset M$ so in particular $f[M\cap X]\subset M\cap f[X]$. For the other inclusion pick $x\in f[X]\cap M$. Using Lemma~\ref{l:unique-M} and the absoluteness of the following formula (and its subformulas)
\[
\exists y\in X\quad  (f(y)=x),\eqno{(*)}
\]
there exists $y\in M\cap X$ with $f(y)=x$ and so $x\in f[M\cap X]$.\\
\eqref{it:supportSubsetOfM}: Pick $x\in X\cap M$. Using Lemma~\ref{l:unique-M} and absoluteness of the following formula (and its subformulas)
\[
\exists D\subset I\quad  (i\in D \Leftrightarrow x(i)\neq 0),\eqno{(*)}
\]
we obtain that $\suppt(x)\in M$. Since $\suppt(x)$ is a countable set, by Lemma~\ref{l:basics-in-M} \eqref{belongIsContained} we obtain that $\suppt(x)\subset M$.

\eqref{it:containsspacecontinuousfunctions}: Using Lemma~\ref{l:unique-M} and absoluteness of the following formulas (and their subformulas) 
\[
\exists \C(X)\in \er^X (\forall f\in \er^X: f\in \C(X)\Leftrightarrow f \text{ is continuous}),\eqno{(*)}
\]
\[
\exists +\in \C(X)^{\C(X)\times \C(X)} (\forall f, g\in \C(X)\;\forall x\in X: + (f,g)(x) = f(x) + g(x)),\eqno{(*)}
\]
\[
\exists \cdot\in \C(X)^{\er\times \C(X)} (\forall \alpha\in\er \forall f\in \C(X)\;\forall x\in X: \cdot (\alpha,f)(x) = \alpha f(x)),\eqno{(*)}
\]
\[
\exists \otimes\in \C(X)^{\C(X)\times \C(X)} (\forall f, g\in \C(X)\;\forall x\in X: \otimes (f,g)(x) = f(x)g(x)),\eqno{(*)}
\]
we obtain that $\C(X)\in M$ and that $\{+,\cdot,\otimes\}\subset M$. Morevoer, if $X$ is a compact space, then using Lemma~\ref{l:unique-M} and the absoluteness of the following formula (and its subformulas)
\[
\exists \|\cdot\|_\infty\in \er^{\C(X)}\quad (\forall f\in\C(X):\; \|\cdot\|(f) = \sup_{x\in X} |f(x)|),\eqno{(*)}
\]
we obtain that $M$ contains the normed linear space $\C(X)$. Thus, by Lemma~\ref{l:basics-in-M} \eqref{linearSubspaceM}, $\overline{\C(X)\cap M}$ is a closed subspace of $\C(X)$ and, since $\otimes\in M$, $\C(X)\cap M$ is closed under multiplication and therefore $\overline{\C(X)\cap M}$ is closed under multiplication as well. Finally, using Lemma~\ref{l:unique-M} and absoluteness of the following formula (and its subformulas)
\[\exists f\in\C(X)\quad (\forall x\in X\;f(x)=1),\eqno{(*)}\]
we obtain that $1\in\C(X)\cap M$.\\
\eqref{it:algebraDenseInFunctionsFromM}: By \eqref{it:containsspacecontinuousfunctions}, $\overline{\C(K)\cap M}$ is a closed subalgebra of $C(K)$ that contains $(\A\cup \{1\})\cap M$, so we have $\overline{\alg((\A\cup \{1\})\cap M)} \subset \overline{\C(K)\cap M}$. For the other inclusion, pick $f\in\C(K)\cap M$. By Lemma~\ref{l:unique-M} and absoluteness of the following formula (and its subformulas)
\[
\exists A\subset \A\; \quad(A\text{ is countable and }f\in \overline{\operatorname{alg}(A\cup \{1\})}),\eqno{(*)}
\]
there is a countable set $A\subset \A$ with $A\in M$ and $f\in \overline{\operatorname{alg}(A\cup \{1\})}$. By Lemma~\ref{l:basics-in-M} \eqref{belongIsContained}, we have that $A\subset \A\cap M$. Therefore, using that $1\in M$, we obtain
\[\overline{\alg((\A\cup \{1\})\cap M)} = \overline{\operatorname{alg}((\A\cap M)\cup \{1\})} \supset \overline{\C(K)\cap M}.\]
\eqref{it:metrizableSubspace}: By \eqref{it:containsspacecontinuousfunctions}, Lemma~\ref{l:unique-M} and the absoluteness of the following formula (and its subformulas)
\[
\exists A\subset \C(K)\quad (A\text{ is countable and separates the points of $K'$}),\eqno{(*)}
\]
there is a countable set $A\subset \C(K)$ with $A\in M$ which separates the points of $K'$. By Lemma~\ref{l:basics-in-M} \eqref{belongIsContained}, we have that $A\subset \C(K)\cap M$ so $\C(K)\cap M$ separates the points of $K'$. Therefore, since by \eqref{it:containsspacecontinuousfunctions} the set $\overline{\C(K)\cap M}$ is a closed algebra containing constant functions, Stone-Weierstrass theorem ensures that the set $\{f|_{K'}\colon f\in \overline{\C(K)\cap M}\}$ is dense in $\C(K')$, which implies that $\{f|_{K'}\colon f\in \C(K)\cap M\}$ is dense in $\C(K')$ and therefore $\{f^{-1}(-1/2,1/2)\cap K'\colon f\in \C(K)\cap M\}$ is an open basis of $K'$. Moreover, for every $f\in\C(K)\cap M$ using Lemma~\ref{l:unique-M} and the absoluteness of the following formula (and its subformulas)
\[
\exists x\in f^{-1}(-1/2,1/2)\cap K',\eqno{(*)}
\]
we have that $f^{-1}(-1/2,1/2)\cap (K'\cap M)\neq \emptyset$ for every $f\in \C(K)\cap M$ and therefore the set $K'\cap M$ is dense in $K'$.\\
\eqref{it:norm}: Since $D\subset K$ is a dense set, we have that $\|f\| = \|f|_D\|$. It follows from \eqref{it:containsspacecontinuousfunctions} that $\|\cdot\|\in M$ and so $\|f\|\in M$. Therefore, using Lemma~\ref{l:unique-M} and the absoluteness of the following formula (and its subformulas) 
\[
\forall n\in\omega\;\exists x\in D\quad (\|f\|-1/n<|f(x)|<\|f\|+1/n),\eqno{(*)}
\]
we obtain that for every $n \in \omega$, there exists $x_n\in D\cap M$ such that $|f(x_n)|\to \|f\|$. 
\end{proof}

\subsection{Retractions associated to suitable models} Here we show that in a compact space with a retractional skeleton, for every suitable model there is a canonical retraction associated to it (see Definition~\ref{def:canonicalRetraction}). The main outcome of this subsection is Theorem~\ref{thm:canonicalRetraction}, where the properties of a canonical retraction are summarized.

Lemma~\ref{l:abstraction} and Lemma~\ref{l:retraction} are inspired by \cite[Lemma 4.7]{CK15}, where something similar was proved for suitable models which are countable.

\begin{lemma}\label{l:abstraction}
For every suitable model $M$ the following holds: Let $X$ be a set and $\A\subset \er^X$ such that $\{X\}\cup \A\subset M$. Consider the mapping $q_M:X\to \er^\A$ defined for $x\in X$ as $q_M(x)(f):=f(x)$, $f\in\A$. Then for every $B\subset X$ with $B\in M$ we have $q_M[B]\subset \overline{q_M[B\cap M]}$.
\end{lemma}
\begin{proof}
In this proof we will use the identification of any $n\in\omega$ with the set $\{0, ..., n-1\}$. Further, denote by $\B$ the set of all the open intervals with rational endpoints and by $\B^{<\omega}$ the set of all the functions whose domain is some $n\in\omega$ and whose values are in $\B$.

Let $S$ be the countable set from the statement of Lemma~\ref{l:basics-in-M} enriched by $\{\B,\B^{<\omega}\}$ and let $\Phi$ be the list of formulas from the statement of Lemma~\ref{l:basics-in-M} enriched by formulas (and their subformulas) marked by $(*)$ in the proof below. Let $M\prec (\Phi; S\cup\{X\}\cup \A)$.

Fix $B\subset X$ with $B\in M$, a point $x\in B$ and a basic neighborhood of a point $q_M(x)$; that is, let us pick finitely many functions $F\subset \A$ and a sequence of rational intervals such that $f(x)\in I_f$, $f\in F$ and consider the neighborhood 
\[
	N:= \{y\in\er^{\A}\setsep y(f)\in I_f\text{ for every }f\in F\}.
\]
By Lemma~\ref{l:basics-in-M} \eqref{finite-in-M}, we have that $F\in M$ and by absoluteness of the formula
\[
\exists n\in\omega \exists\eta\quad (\eta\text{ is a bijection between $n$ and $F$}),\eqno{(*)}
\]
and its subformulas, there is $n\in\omega$ and a bijection $\eta\in M$ between $n$ and $F$. 

Let us further define the mapping $\xi:n\to\B$ by $\xi(i) =I_{\eta(i)}$. Since $\xi\in \B^{<\omega}\in M$, it follows from Lemma~\ref{l:basics-in-M} \eqref{belongIsContained} that $\xi\in M$.
By Lemma~\ref{l:unique-M} and the absoluteness of the formula (and its subformulas)
\[
\exists x\in B \quad(\forall i<n: \eta(i)(x)\in \xi(i)),\eqno{(*)}
\]
there is a point $x_0\in B\cap M$ such that $q_M(x_0)\in N$; hence, $q_M[B\cap M]$ is dense in $q_M[B]$.
\end{proof}

\begin{lemma}\label{l:retraction}For every suitable model $M$ the following holds: Let $(K, \tau)$ be a compact space and $D\subset K$ be a dense subset with $\{K,D,\tau\}\subset M$.

If $\C(K)\cap M$ separates the points of $\overline{D\cap M}$, then there exists a unique retraction $r_M:K\to \overline{D\cap M}$ such that $f = f\circ r_M$, for every $f\in\C(K)\cap M$.

Moreover, in this case
\begin{enumerate}
    \item for every $x\in K$ and $\A\subset \C(K)$ separating the points of $K$ with $\A\in M$, $r_M(x)$ is the unique point from $\overline{D\cap M}$ satisfying $f(r_M(x)) = f(x)$, for every $f\in \A\cap M$. 
    \item\label{it: imageofclosure} if $B\subset D$ and $B\in M$, then $r_M[\overline{B}] = \overline{B\cap M}$.
\end{enumerate}
\end{lemma}
\begin{proof} Let $S$ and $\Phi$ be the union of sets and the lists of formulas from the statements of Lemma~\ref{l:basics-in-M}, Lemma~\ref{l:basics-in-M-2} and Lemma~\ref{l:abstraction}. Let $M\prec (\Phi; S\cup\{K,D,\tau\})$ be such that $\C(K)\cap M$ separates the points of $\overline{D\cap M}$. By Lemma~\ref{l:basics-in-M-2} \eqref{it:containsspacecontinuousfunctions}, we have that $\C(K)\in M$.

Let us consider the mapping $q_M:K\to \er^{\C(K)\cap M}$ given by $q_M(x) = (f(x))_{f\in\C(K)\cap M}$, $x\in K$. Then $q_M$ is continuous and, by the assumption, $q_M|_{\overline{D\cap M}}$ is one-to-one; hence, $q_M|_{\overline{D\cap M}}$ is a homeomorphic embedding. Moreover, whenever $B\subset D$ is such that $B\in M$, then by Lemma~\ref{l:abstraction} we have $\overline{q_M[B\cap M]}\supset q_M[B]$ which implies $q_M[\overline{B}] = q_M[\overline{B\cap M}]$.

Now, put $r_M:= (q_M|_{\overline{D\cap M}})^{-1}\circ q_M$. Then it is a continuous retraction with $r_M[K] = \overline{D\cap M}$. Moreover, for every $x\in K$
\[
r_M(x) = (q_M|_{\overline{D\cap M}})^{-1}\circ q_M(x) = y,
\]
where $y\in K$ is the unique point such that $y\in\overline{D\cap M}$ and $g(y) = g(x)$, for every $g\in \C(K)\cap M$. Hence, for $f\in \C(K)\cap M$ we have
\[
f(r_M(x)) = f(y) = f(x).
\]

In order to see that $r_M$ is unique, let us consider another retraction $r':K \to \overline{D \cap M}$ satisfying that $f = f\circ r'$ for every $f\in\C(K)\cap M$. Then, for every $x\in K$, and every $f\in \C(K)\cap M$ we have $f(r_M(x)) = f(x) = f(r'(x))$; hence, since $\C(K)\cap M$ separates the points of $r_M[K]$, it holds that $r_M(x) = r'(x)$. Since $x\in K$ was arbitrary, we have $r_M = r'$.
Moreover, given $y\in \overline{D\cap M}$ such that $f(y)=f(x)$ for every $f\in\A\cap M$, where $\A\subset \C(K)$ is a set separating the points of $K$ with $\A\in M$, we obtain that $f(y)=f(x)$ for $f\in \overline{\alg((\A\cup \{1\})\cap M)} = \overline{\C(K)\cap M}$ (the last equality follows from Lemma~\ref{l:basics-in-M-2} \eqref{it:algebraDenseInFunctionsFromM}) and so $y=r_M(x)$.

Finally, if $B\subset D$ is such that $B\in M$ then by the above we have $q_M[\overline{B}] = q_M[\overline{B\cap M}]$ and so $r_M[\overline{B}] = r_M[\overline{B\cap M}] = \overline{B\cap M}$.
\end{proof}

Let us note that a compact space $K$ admits a retractional skeleton if and only if there exists a dense set $D\subset K$ such that for every suitable model $M$ which is moreover countable, the set $\C(K)\cap M$ separates the points of $\overline{D\cap M}$, see e.g. \cite[Theorem 4.9]{C14} or \cite[Theorem 19.16]{KKLbook} (that is, the assumption of Lemma~\ref{l:retraction} is satisfied for suitable models which are countable). The following shows that we do not need to assume countability of the model.

\begin{prop}\label{p:characterizationRSkeletonNotCountableModels}
For every suitable model $M$ the following holds: If $(K,\tau)$ is a compact space and $D\subset K$ is a subset of a set induced by a retractional skeleton with $\{D,K,\tau\}\subset M$, then $\C(K)\cap M$ separates the points of $\overline{D\cap M}$.
\end{prop}
\begin{proof}
Let $S$ be the union of sets from the statements of Lemma~\ref{l:basics-in-M} and Lemma~\ref{l:basics-in-M-2} and let $\Phi$ be the union of lists of formulas from the statements of Lemma~\ref{l:basics-in-M} and Lemma~\ref{l:basics-in-M-2} enriched by formulas (and their subformulas) marked by $(*)$ in the proof below. Let $M\prec (\Phi; S\cup \{K,D,\tau\})$.

By Lemma~\ref{l:unique-M} and the absoluteness of the following formula (and its subformulas)
\[\begin{split}
\exists \Gamma\; \exists \le\;\exists r\; \big(& D \text{ is a subset of a set induced by}\\ & \text{the retractional skeleton }\{r(s)\colon s\in\Gamma\}\big),\qquad (*)
\end{split}\]
there exist $\Gamma, \le, r\in M$ such that $\{r(s)\colon s\in \Gamma\}$ is a retractional skeleton on $K$ inducing a set containing $D$. For $s\in\Gamma$ we will write below $r_s$ instead of $r(s)$. By Lemma~\ref{l:basics-in-M-2} \eqref{it:upDirected}, the set $\Gamma\cap M$ is up-directed. Hence by Lemma~\ref{l: retraction from up-directed subset} there exists a continuous retraction $R_M\colon K\to K$ defined by $R_{M}(x):=\lim_{s\in \Gamma\cap M}r_{s}(x)$, for every $x\in K$. Using the absoluteness of the following formula (and its subformulas)
\[
\forall u \in D\,\exists s\in\Gamma\, u\in r_s[K],\eqno{(*)}
\]
we obtain that $D\cap M\subset R_{M}[K]$ and therefore $\overline{D\cap M}\subset R_{M}[K]$. Now fix $x,y\in\overline{D\cap M}$ with $x\neq y$. Since $x=\lim_{s\in \Gamma\cap M}r_s(x)$ and $y=\lim_{s\in\Gamma\cap M}r_{s}(y)$ there exists $s\in \Gamma\cap M$ such that $r_{s}(x)\neq r_s(y)$. By Lemma~\ref{l:basics-in-M} \eqref{it: Dom-Rng-in-M}, we have that $r(s)=r_s\in M$ and $r_s[K]\in M$. Thus, by Lemma~\ref{l:basics-in-M-2} \eqref{it:metrizableSubspace}, there exists $f\in \C(K)\cap M$ such that $f(r_s(x))\neq f(r_s(y))$. Now using Lemma~\ref{l:basics-in-M-2} \eqref{it:composition}, we obtain that $g=f\circ r_s\in \C(K)\cap M$ and $g(x)\neq g(y)$. Thus, $\C(K)\cap M$ separates the points of $\overline{D\cap M}$
\end{proof}

The retraction constructed in Lemma~\ref{l:retraction} (whose assumption is satisfied by Proposition~\ref{p:characterizationRSkeletonNotCountableModels} in compact spaces admitting a retractional skeleton) will be the key to our considerations. Let us give it a name.

\begin{defin}\label{def:canonicalRetraction}
Let $K$ be a compact space and let $D\subset K$ be a dense subset that is contained in the set induced by a retractional skeleton. Given a set $M$, we say that $r_M$ is the \emph{canonical retraction associated to $M$, $K$ and $D$} if it is the unique retraction on $K$ satisfying $r_M[K]=\overline{D\cap M}$ and $f = f\circ r_M$, for every $f\in \C(K)\cap M$. We say that a set $M$ \emph{admits canonical retraction} if there exists the canonical retraction associated to $M$, $K$ and $D$.

In the case when $D=K$ we say that $M$ \emph{admits canonical retraction $r_M$ associated to $M$ and $K$}.
\end{defin}

The properties of canonical retractions associated to suitable models are summarized in Theorem~\ref{thm:canonicalRetraction}. We need two lemmas first.

\begin{lemma}\label{l:separation}For every suitable model $M$ the following holds: Let $(K,\tau)$ be a compact space and $\A\subset\C(K)$ be a set separating the points of $K$ with $\{\A,K,\tau\}\subset M$. Then for every compact set $K'\subset K$, $\A\cap M$ separates the points of $K'$ if and only if $\C(K)\cap M$ separates the points of $K'$.
\end{lemma}
\begin{proof}Let $S$ and $\Phi$ be the set and the list of formulas from the statements of Lemma~\ref{l:basics-in-M} and Lemma~\ref{l:basics-in-M-2}. Let $M\prec (\Phi; S\cup\{K,\tau,\A\})$. 
In order to get a contradiction, let us assume that $\C(K)\cap M$ separates the points of $K'$ but $\A\cap M$ does not separate the points of $K'$. Then also $\overline{\operatorname{alg}((\A\cap M)\cup \{1\})}$ does not separate the points of $K'$ (because if there are $x\neq y$ with $f(x) = f(y)$ for every $f\in \A\cap M$, then also $g(x) = g(y)$ for every $g$ of the form $g = a_0 + \sum_{i=1}^n a_i\Pi_{j=1}^m f_{i,j}$). But this is a contradiction, because using Lemma~\ref{l:basics-in-M-2} \eqref{it:containsspacecontinuousfunctions} and \eqref{it:algebraDenseInFunctionsFromM} we conclude that $\overline{\operatorname{alg}((\A\cap M)\cup \{1\})} = \overline{\C(K)\cap M}$.
\end{proof}

\begin{lemma}\label{l:dense}For every suitable model $M$ the following holds: Let $(K,\tau)$ be a compact space and let $D\subset K$ be a dense subset that is contained in the set induced by a retractional skeleton such that $\{K,D,\tau\}\subset M$. Then the mapping $\Phi:\overline{\C(K)\cap M}\to \C(\overline{D\cap M})$ defined by $\Phi(f):=f|_{\overline{D\cap M}}$, for every $f\in \overline{\C(K)\cap M}$, is a surjective isometry.
\end{lemma}
\begin{proof}
Let $S$ and $\Phi$ be the union of countable sets and finite lists of formulas from the statements of Lemma~\ref{l:basics-in-M-2} and Proposition~\ref{p:characterizationRSkeletonNotCountableModels}. Let $M\prec (\Phi; S\cup \{K,D,\tau\})$. By Lemma~\ref{l:basics-in-M-2} \eqref{it:norm}, we have that $\|f\| = \|f|_{D\cap M}\|$, for every $f\in \C(K)\cap M$, so the mapping $\Phi|_{\C(K)\cap M}$ is an isometry which implies that $\Phi$ is also an isometry. It remains to show that it is surjective. By Lemma~\ref{l:basics-in-M-2} \eqref{it:containsspacecontinuousfunctions}, $\overline{\C(K)\cap M}$ is a closed subalgebra of $\C(K)$ and so the image of $\Phi$ is a closed subalgebra of $\C(\overline{D\cap M})$ which, by Proposition~\ref{p:characterizationRSkeletonNotCountableModels} separates the points of $\overline{D\cap M}$. Therefore, it follows from Stone-Weierstrass theorem that $\Phi[\overline{\C(K)\cap M}] = \C(\overline{D\cap M})$.
\end{proof}

\begin{thm}\label{thm:canonicalRetraction}
For every suitable model $M$ the following holds: Let $(K,\tau)$ be a compact space and let $D\subset K$ be a dense subset that is contained in the set induced by a retractional skeleton with $\{K,D,\tau\}\subset M$. Then there exists a unique retraction $r_M:K\to \overline{D\cap M}$ with $r_M[K] = \overline{D\cap M}$ and $f=f\circ r_M$, for every $f\in\C(K)\cap M$. Moreover, for this retraction $r_M$ the following holds:
\begin{enumerate}[label = (\roman*)]
    \item\label{it:characterizationFirst} Whenever $\A\subset \C(K)$ separates the points of $K$ and $\A\in M$, then for every $x\in K$, $r_M(x)$ is the unique point from $\overline{D\cap M}$ satisfying $f(r_M(x)) = f(x)$, for every $f\in \A\cap M$. 
    \item\label{it:characterizationSecond} Whenever $(\Gamma,\leq)$ is up-directed and $\sigma$-complete and $r:\Gamma\to\C(K,K)$ is a mapping such that $\mathfrak{s} = \{r(s)\colon s\in\Gamma\}$ is a retractional skeleton on $K$ with $D\subset D(\mathfrak{s})$ and $\{r,\Gamma,\leq\}\subset M$, then the following holds.
    \begin{enumerate}[label=(\alph*)]
        \item\label{it:imageOfRs} For every $s\in \Gamma\cap M$, $r(s)[K]\subset \overline{D\cap M}$.
        \item\label{it:formulaUsingSkeleton} \[r_M(x) = \lim_{s\in\Gamma\cap M} r(s)(x),\quad x\in K.
    \]
    \item\label{it:countableModelsAreInSkeleton} If $M$ is countable, then $r_M = r(s)$ for $s = \sup \Gamma\cap M$.
    \item\label{it:skeletonOnTheInducedSubspace} $\{r(s)|_{r_M[K]}\colon s \in (\Gamma\cap M)_\sigma\}$ is a retractional skeleton on $r_M[K]$ with induced set $D(\mathfrak{s})\cap r_M[K]$.
    \item\label{it:commutativeCaseModels} If $\mathfrak{s}$ is commutative, then $r_M[D(\mathfrak{s})] = D(\mathfrak{s})\cap r_M[K]$.
        \end{enumerate}
    \item\label{it:homeomorphicImageOfCanonicalRetraction} Whenever $h:K\to L$ is a surjective homeomorphism with $h\in M$, then $r:=h\circ r_M\circ h^{-1}$ is the unique retraction on $L$ such that $r[L] = \overline{h[D]\cap M}$ and $f\circ r = f$, for every $f\in \C(L)\cap M$.
    \item\label{it:weight} $\w(r_M[K])\leq |M|$.
\end{enumerate}
\end{thm}
\begin{proof}Denote by $\B$ the set of all open intervals with rational endpoints. Let $S$ be the union of sets from the statements of Lemma~\ref{l:basics-in-M}, Lemma~\ref{l:basics-in-M-2}, Lemma~\ref{l:retraction}, Lemma~\ref{l:dense} and Proposition~\ref{p:characterizationRSkeletonNotCountableModels} and let $\Phi$ be the union of lists of formulas from the statements of Lemma~\ref{l:basics-in-M}, Lemma~\ref{l:basics-in-M-2}, Lemma~\ref{l:retraction}, Lemma~\ref{l:dense} and Proposition~\ref{p:characterizationRSkeletonNotCountableModels} enriched by the formula (and its subformulas) marked by $(*)$ in the proof below. Let $M\prec (\Phi; S\cup \{K,D,\tau\})$. It follows directly from Proposition~\ref{p:characterizationRSkeletonNotCountableModels} and Lemma~\ref{l:retraction} that the retraction $r_M$ exists and that it satisfies \ref{it:characterizationFirst}.

Let $\{r,\Gamma,\leq\}\subset M$ be as in \ref{it:characterizationSecond}. For $s\in\Gamma\cap M$, by Lemma~\ref{l:basics-in-M} \eqref{it: Dom-Rng-in-M} we have that $r(s)[K] \in M$ and thus Lemma~\ref{l:basics-in-M-2} \eqref{it:metrizableSubspace} ensures that $r(s)[K]\subset \overline{r(s)[K]\cap M}$. Moreover, for every $x\in r(s)[K]\cap M$, using the countable tightness of $D(\mathfrak{s})$ the following formula holds
\[\exists C\subset D\;\quad ( \text{$C$ is countable and }x\in \overline{C}).\eqno{(*)}\] Thus, by Lemma~\ref{l:unique-M}  there exists a countable set $C\subset D$ with $C\in M$ (which implies $C\subset M$) such that $x\in \overline{C}$, which implies that $x\in \overline{D\cap M}$, so $r(s)[K]\subset \overline{r(s)[K]\cap M}\subset \overline{D\cap M}$ and \ref{it:imageOfRs} holds. By Lemma~\ref{l:basics-in-M-2} \eqref{it:upDirected}, $\Gamma\cap M$ is up-directed and so, using Lemma~\ref{l: retraction from up-directed subset} the limit $\lim_{s\in\Gamma\cap M} r(s)(x) =:R_M(x)$ exists for every $x\in K$.

We \emph{claim} that $R_M=r_M$. Note that due to the uniqueness of $r_M$, it is enough to show that $R_M[K] \subset \overline{D \cap M}$ and that $f\circ R_M = f$, for every $f\in\C(K)\cap M$. By the definition of $R_M$, using \ref{it:imageOfRs}, we observe that $R_M[K] \subset \overline{D \cap M}$.   Moreover, for every $x\in D\cap M$, by Lemma~\ref{l:unique-M} and the absoluteness of the following formula (and its subformulas)
\[\exists s\in\Gamma\; \big(r_s(x)=x\big),\eqno{(*)}\]
there is $s\in\Gamma\cap M$ with $r_s(x)=x$ and so $x\in R_M[K]$. Thus, we have that $R_M[K] = \overline{D\cap M}$. Pick $f\in \C(K)\cap M$. Since $(r_s)_{s\in\Gamma\cap M}$ converges pointwise to $R_M$, using \cite[Lemma 5.2]{K20} we conclude that $(f\circ r_s)_{s\in\Gamma\cap M}$ converges in norm to $f\circ R_M$. Therefore $f\circ R_M\in \overline{\C(K)\cap M}$, since it follows from Lemma~\ref{l:basics-in-M} \eqref{it: Dom-Rng-in-M} and Lemma~\ref{l:basics-in-M-2} \eqref{it:composition} that $f\circ r_s\in \C(K) \cap M$, for every $s\in\Gamma\cap M$.
Thus $f$ and $f\circ R_M$ are two functions from $\overline{\C(K)\cap M}$ which have the same values on $\overline{D\cap M}$ and so, by Lemma~\ref{l:dense}, $f = f\circ R_M$. This proves the claim and establishes \ref{it:formulaUsingSkeleton} which, using Lemma~\ref{l: retraction from up-directed subset}, implies \ref{it:countableModelsAreInSkeleton}, \ref{it:skeletonOnTheInducedSubspace} and \ref{it:commutativeCaseModels}.

The proof of \ref{it:homeomorphicImageOfCanonicalRetraction} is easy once we realize that by Lemma~\ref{l:basics-in-M-2} \eqref{it:one-to-one-imageOfM} we have that $h[D\cap M] = h[D]\cap M$ and that $f\circ h\in \C(K)\cap M$, for every $f\in \C(L)\cap M$. We omit the straightforward details. 

Finally, \ref{it:weight} follows from Proposition~\ref{p:characterizationRSkeletonNotCountableModels}, since $r_M[K]=\overline{D\cap M}$.
\end{proof}

\subsection{Families of canonical retractions} Here we study families of canonical retractions associated to suitable models. Those are more-or-less straightforward consequences of Theorem~\ref{thm:canonicalRetraction}. The most important for what follows is Proposition~\ref{prop:rri} which will be repeatedly used further.

\begin{lemma}\label{l:propertiesOfCanonicalRetractions}There are a countable set $S$ and a finite list of formulas $\Phi$ such that the following holds: Let $(K, \tau)$ be a compact space, $(\Gamma,\leq)$ be an up-directed set, $r:\Gamma\to\C(K,K)$ be a mapping such that $\mathfrak{s}:=\{r(s)\colon s\in\Gamma\}$ is a retractional skeleton on $K$ and let $D\subset D(\mathfrak{s})$ be dense in $K$. Put $S'=S\cup\{K, \tau, D,\Gamma,\leq,r\}$. Then every $M\prec(\Phi; S')$ admits the canonical retraction $r_M$ associated to $M$, $K$ and $D$.

Moreover, we have the following.
\begin{enumerate}
    \item\label{it:commutativity} If $M, N\prec(\Phi; S')$ and $M\subset N$ then $r_M\circ r_N = r_N\circ r_M = r_M$.
    \item\label{it:increasingSequence} Let $\M$ be an up-directed set with $M\prec(\Phi; S')$, for every $M\in \M$ and let $M_\infty:=\bigcup_{M\in\M} M$. Then $M_\infty\prec (\Phi; S')$ and $\lim_{M\in\M}r_{M}(x) = r_{M_\infty}(x)$, $x\in K$.
    \item\label{it:toIdentity} If $\U$ is a basis of $\tau$, $\M$ is an up-directed set with $M\prec(\Phi; S')$, for every $M\in \M$ and $\U\subset\bigcup_{M\in\M} M$, then $\lim_{M\in\M}r_{M}(x) = x$, for every $x\in K$.
\end{enumerate}
\end{lemma}
\begin{proof}
The existence of $S$ and $\Phi$ follows from Theorem~\ref{thm:canonicalRetraction}. Let us prove the moreover part using the additional properties of canonical retractions established in Theorem~\ref{thm:canonicalRetraction}.\\
\eqref{it:commutativity} Since $r_M[K] = \overline{D\cap M}\subset \overline{D\cap N}=r_N[K]$, we have $r_M = r_N\circ r_M$. Moreover, for every $f\in\C(K)\cap M\subset \C(K)\cap N$ we have 
    \[f(r_{M}(x)) = f(x) = f(r_N(x)) = f(r_M(r_N(x))),\quad x\in K\]
    and so $r_M(x) = r_M(r_N(x))$, for every $x\in K$.\\
\eqref{it:increasingSequence} Since $\M$ is up-directed, it follows from \cite[Lemma~2.1]{C12} and Lemma~\ref{l:unique-M} that $M_\infty\prec (S';\Phi)$.  Now, combining Theorem~\ref{thm:canonicalRetraction} \ref{it:characterizationSecond} with Lemma \ref{l:basics-in-M-2} \eqref{it:upDirected} and Lemma~\ref{l: retraction from up-directed subset} \ref{it:generalUpDirected} we obtain that $\lim_{M\in\M}r_{M}(x) = r_{M_\infty}(x)$, $x\in K$.\\
\eqref{it:toIdentity} Pick $x\in K$, $U\in\U$ such that $x\in U$ and find $V \in\U$ with $x\in V\subset \overline{V}\subset U$. Since $V \in \U$ and $\U \subset \bigcup_{M \in \M} M$, there exists $M_0 \in \M$ such that $V \in M_0$. Now fix $M \in \M$ with $M_0 \subset M$. It follows from Lemma~\ref{l:basics-in-M} \eqref{operations-sets-in-M} that $V \cap D \in M$. Therefore Lemma~\ref{l:retraction} ensures that $r_M(x)\in r_M[\overline{V\cap D}]\subset \overline{V\cap D}\subset U$, since $x\in \overline{V\cap D}$.
\end{proof}

\begin{prop}\label{prop:rri}There exist a countable set $S$ and a finite list of formulas $\Phi$ such that the following holds:\\
Let $(K,\tau)$ be a compact space, $(\Gamma,\leq)$ be an up-directed set and $r:\Gamma\to\C(K,K)$ be a mapping such that $\mathfrak{s}=\{r(s)\colon s\in\Gamma\}$ is a retractional skeleton on $K$. Let $\kappa:=\w(K)$ and $\U:\kappa\to\tau$ be such that $\{\U(i)\colon i<\kappa\}$ is an open basis of $\tau$. Put $S':=S\cup \{K,D(\mathfrak{s}),\Gamma,\leq,r,\tau,\U\}$. Let $(M_\alpha)_{\alpha\leq \kappa}$ be a sequence of sets satisfying
\begin{enumerate}[label=(R\alph*)]
    \item\label{it:areModels} $M_\alpha\prec (\Phi;S')$, for every $\alpha\in [0,\kappa]$,
    \item\label{it:cardinalityMalpha} $\vert M_\alpha \vert \le \max(\omega, \vert \alpha \vert)$, for every $\alpha\in [0,\kappa]$,
    \item\label{it:areIncreasing} $M_{\alpha+1}\supset M_\alpha\cup\{\alpha\}$, for every $\alpha\in [0,\kappa)$,
    \item\label{it:areContinuous} $M_\alpha = \bigcup_{\beta<\alpha} M_\beta$, if $\alpha\in(0,\kappa]$ is a limit ordinal.
\end{enumerate}
Then for every $\alpha\in[0,\kappa]$ there exists a canonical retraction $r_\alpha$ associated to $M_\alpha$, $K$ and $D(\mathfrak{s})$ and the following holds.

\begin{enumerate}[label=(R\arabic*)] 
\newcounter{saveenum} 
    \item\label{it:comutative} For every $\alpha<\beta$, we have that $r_\alpha\circ r_\beta = r_\beta\circ r_\alpha = r_\alpha$.
    \item\label{it:separatePoints} $r_\alpha(x)\to x$, for every $x\in K$.
     \item\label{it:upDirectedLimit} Let $\alpha\le\kappa$, let $\eta:[0,\alpha)\to \kappa$ be an increasing function and let $\xi \le\kappa$ be a limit ordinal with $\sup_{\beta<\alpha} \eta(\beta) = \xi$. Then $\lim_{\beta < \alpha} r_{\eta(\beta)}(x) = r_\xi(x)$, for every $x\in K$.
     \item\label{it:identity} $r_\kappa = id$.
     \item\label{it:skeletonOnSubset} For every $\alpha\in[0,\kappa]$, we have that $\w(r_\alpha[K])\leq max(\omega, |\alpha|)$ and $(r_s|_{r_\alpha[K]})_{s\in (\Gamma\cap M_\alpha)_\sigma}$ is a retractional skeleton on $r_\alpha[K]$ with induced set $D(\mathfrak{s})\cap r_\alpha[K]$.
     \item\label{it:compositionWithRa} If $\A$ is a closed subset of $C(K)$ and $f \circ r_s \in \A$, for every $f \in \A$ and $s\in\Gamma$, then $\{f \circ r_\alpha\colon f\in\A, \alpha\le\kappa\}\subset \A$.
      \item\label{it:ordinalSeparatingPoints} If $x, y$ are distinct points in $K$ then $\beta:=\min\{\alpha< \kappa\colon r_\alpha(x)\neq r_\alpha(y)\}$ exists and it is a successor ordinal or $\beta=0$.
      \item\label{it:secondFormula} For every $\alpha\leq \kappa$, the set $\Gamma\cap M_\alpha$ is up-directed and $r_\alpha(x) = \lim_{s\in (M_\alpha\cap \Gamma)} r_s(x)$, for every $x \in K$.
      \item \label{comutaralpha3} For every $\alpha \le \kappa$ and $t \in (\Gamma \cap M_\alpha)_\sigma$, it holds that $r_t \circ r_\alpha=r_\alpha \circ r_t$.
      \item \label{Model-compositionRalpha} Let $\A \subset \C(K)$ be a set that separates the points of $K$ and $\alpha < \kappa$. If $\A \in M_\alpha$,  then $\A \cap M_\alpha$ separates the points of $r_\alpha[K]$ and $f \circ r_\alpha=f$, for every $f \in \A \cap M_\alpha$.
      \setcounter{saveenum}{\value{enumi}} 
\end{enumerate}
Moreover, if $\mathfrak{s}$ is full or commutative, then $r_\alpha[D(\mathfrak{s})]\subset D(\mathfrak{s})$, for every $\alpha\in [0,\kappa]$ and if $D\subset D(\mathfrak{s})$ is such that $r_\alpha[D]\subset D$ for every $\alpha\in [0,\kappa]$, then we also have the following.
\begin{enumerate}[label=(R\arabic*)]
 \setcounter{enumi}{\value{saveenum}} 
    \item\label{it:countableSupport2} The sets $\{r_\alpha(x)\colon \alpha<\kappa\}$ and $\{\alpha<\kappa: r_{\alpha}(x)\neq r_{\alpha+1}(x)\}$ are countable, for every $x\in D$.
\end{enumerate}
\end{prop}
\begin{proof}
Let $S$ and $\Phi$ be the union of sets and lists of formulas from the statements of Theorem~\ref{thm:canonicalRetraction} and Lemma~\ref{l:propertiesOfCanonicalRetractions}. Then the existence of $r_\alpha$, $\alpha\in [0,\kappa]$ follows from Theorem~\ref{thm:canonicalRetraction}. Now, \ref{it:comutative} follows immediately from Lemma~\ref{l:propertiesOfCanonicalRetractions}. \ref{it:separatePoints} and \ref{it:upDirectedLimit} follow from Lemma~\ref{l:propertiesOfCanonicalRetractions} as well (using for \ref{it:separatePoints} the fact that $\{\U(i)\colon i<\kappa\}\subset \bigcup_{\alpha<\kappa} M_\alpha$ and for \ref{it:upDirectedLimit} the fact that $\bigcup_{\beta<\alpha} M_{\eta(\beta)} = M_\xi$). \ref{it:identity} follows from \ref{it:separatePoints} and \ref{it:upDirectedLimit} applied to $\eta(i):=i$, $i<\kappa$ and $\xi = \kappa$. \ref{it:skeletonOnSubset} and \ref{Model-compositionRalpha} follow from Theorem~\ref{thm:canonicalRetraction}. 
For \ref{it:compositionWithRa} we observe that by Theorem~\ref{thm:canonicalRetraction} \ref{it:characterizationSecond} the net of continuous retractions $(r_s)_{s \in \Gamma \cap M_\alpha}$ converges pointwise to the continuous retraction $r_\alpha$ and so \cite[Lemma~5.2]{K20} ensures that the net $(f \circ r_s)_{s \in \Gamma \cap M_\alpha}$ converges in norm to $f \circ r_\alpha$, for every $f \in \C(K)$, which implies \ref{it:compositionWithRa}. For \ref{it:ordinalSeparatingPoints} we observe that by \ref{it:separatePoints} there is $i<\kappa$ such that $r_i(x)\neq r_i(y)$ so $\beta$ is well defined and if $\beta\neq 0$ then it is a successor ordinal by \ref{it:upDirectedLimit}. For \ref{it:secondFormula} is suffices to apply Lemma \ref{l:basics-in-M-2} \eqref{it:upDirected} and Theorem~\ref{thm:canonicalRetraction} \ref{it:characterizationSecond}. \ref{comutaralpha3} follows from Lemma~\ref{l: retraction from up-directed subset} \ref{it:sigmaClosure}.

Moreover, if $\mathfrak{s}$ is full then we obviously have $r_\alpha[D(\mathfrak{s})] = r_\alpha[K]\subset K=D(\mathfrak{s})$ and if it is commutative then Theorem~\ref{thm:canonicalRetraction} \ref{it:characterizationSecond} ensures that $r_\alpha[D(\mathfrak{s})]\subset D(\mathfrak{s})$.\\
\ref{it:countableSupport2}: Pick $x\in D$ and note that in order to prove that $\{r_\alpha(x)\colon \alpha<\kappa\}$ is countable, it suffices to show that for every strictly increasing function $\eta:[0,\omega_1)\to\kappa$, there is $\zeta<\omega_1$ with $r_{\eta(\zeta)}(x)=r_{\eta(\beta)}(x)$, for every $\zeta<\beta<\omega_1$. 
Let $\eta:[0,\omega_1)\to\kappa$ be a strictly increasing function and set $\xi:=\sup_{\beta<\omega_1} \eta(\beta)$. By \ref{it:upDirectedLimit}, we have $r_\xi(x) = \lim_{\beta < \omega_1} r_{\eta(\beta)}(x)$. Hence, since $r_\xi[D]\subset D$ and $D$ has countable tightness (see \cite[Theorem 32]{kubis09}), there is a $\zeta<\omega_1$ with $r_\xi(x) \in r_{\eta(\zeta)}[K]$; so, for $\zeta\leq\beta<\omega_1$, using \ref{it:comutative}, we obtain $r_{\eta(\beta)}(x) = r_{\eta(\beta)}(r_\xi(x)) = r_\xi(x)$. Finally, to conclude that the set $\{\alpha<\kappa: r_{\alpha}(x)\neq r_{\alpha+1}(x)\}$ is countable, note that the mapping $\varphi(\alpha)=r_\alpha(x)$ is an injection from this set into $\{r_\alpha(x)\colon \alpha<\kappa\}$. Indeed, suppose by contradiction that there exist $\alpha, \beta<\kappa$ with $\alpha \ne \beta$, $r_\alpha(x)\neq r_{\alpha+1}(x)$ and $r_\beta(x)\neq r_{\beta+1}(x)$ such that $r_\alpha(x)=r_\beta(x)$. Without loss of generality, we may assume that $\alpha<\beta$. Then applying the map $r_{\alpha+1}$, by \ref{it:comutative}, we obtain $r_{\alpha}(x)=r_{\alpha+1}(r_{\alpha}(x))= r_{\alpha+1}(r_\beta(x))=r_{\alpha+1}(x)$, which is a contradiction.
\end{proof}

\subsection{Application - passing to a subskeleton} Here we introduce the notion of a (weak) subskeleton and show that for a countable family of retractional skeletons inducing the same set there is a common weak subskeleton, see Theorem~\ref{thm:subskeletons}.

\begin{defin}\label{def:subskeleton}
    Let $K$ be a compact space and let $\mathfrak{s} = (r_s)_{s\in \Gamma}$ be a retractional skeleton on $K$. We say that $(r_s)_{s\in\Gamma'}$ is a \emph{subskeleton} of $\mathfrak{s}$, if $\Gamma'\subset \Gamma$ is a $\sigma$-closed and cofinal subset.
    \end{defin}
It is easy to see that every subskeleton is a retractional skeleton.

\begin{defin}\label{def:weakSubskeleton}
Let $(r_s)_{s \in \Gamma}$ be a retractional skeleton on a compact space $K$. We say that $(R_i)_{i\in\Lambda}$ is a \emph{weak subskeleton of $(r_s)_{s \in \Gamma}$} if $(R_i)_{i\in\Lambda}$ is a retractional skeleton on $K$ and there exists a mapping $\phi:\Lambda\to \Gamma$ such that
\begin{itemize}
    \item $R_i = r_{\phi(i)}$, for every $i\in\Lambda$;
    \item $\phi$ is $\omega$-monotone, that is, if $i,j\in\Lambda$ with $i\leq j$, then $\phi(i)\leq\phi(j)$ and if $(i_n)_{n\in\omega}$ is an increasing sequence from $\Lambda$, then $\sup_n \phi(i_n) = \phi(\sup_n i_n)$;
    \item $\{\phi(i)\colon i\in\Lambda\}$ is cofinal in $\Gamma$.
\end{itemize}
\end{defin}

Clearly, every subskeleton of a retractional skeleton is also a weak subskeleton. Some basic properties of weak subskeletons are summarized below.

\begin{fact}
Let $(r_s)_{s \in \Gamma}$ be a retractional skeleton on a compact space $K$ and $(R_i)_{i\in\Lambda}$ be a weak subskeleton of $(r_s)_{s \in \Gamma}$. Then
\begin{itemize}
    \item $(R_i)_{i\in\Lambda}$ induces the same subset as $(r_s)_{s\in\Gamma}$;
    \item if $(S_j)_{j\in\Delta}$ is a weak subskeleton of $(R_i)_{i\in\Lambda}$, then it is a weak subskeleton of $(r_s)_{s \in \Gamma}$.
\end{itemize}
\end{fact}

The following result shows that we may concentrate properties of countably many retractional skeletons into one skeleton, which is moreover generated by suitable models.

\begin{thm}\label{thm:subskeletons}
Let $K$ be a compact space and let $(r_s^n)_{s\in\Gamma_n}$, $n\in\omega$ be a sequence of retractional skeletons on $K$ inducing the same set $D$. Then there exists a retractional skeleton which is a weak subskeleton of $(r_s^n)_{s\in\Gamma_n}$, for every $n\in\omega$.

Moreover, for every countable set $S$ and every finite list of formulas $\Phi$, there exists a family $\M$ consisting of countable suitable models for $\Phi$ containing $S$ such that every $M\in\M$ admits canonical retraction $r_M$ associated to $M$, $K$ and $D$ and $(r_M)_{M\in\M}$ is a weak subskeleton of $(r_s^n)_{s\in\Gamma_n}$, for every $n\in\omega$, where the ordering on $\M$ is given by inclusion. 
\end{thm}
\begin{proof}Let $\Gamma_n = (\Gamma_n,\leq_n)$ and $r_n:\Gamma_n\to \C(K,K)$ be such that $r_n(s):=r_s^n$, for every $n\in\omega$ and $s\in\Gamma_n$. Let $\tau$ be the topology on $K$. Let $S'$ be the union of $S$ and the countable set from the statement of Theorem~\ref{thm:canonicalRetraction} enriched by $\{K,D,\tau,r_n,\Gamma_n,\leq_n\colon n\in\omega\}$ and let $\Phi'$ be the union of $\Phi$ and the list of formulas from the statement of Theorem~\ref{thm:canonicalRetraction}.
By Theorem~\ref{thm:modelExists}, there is a set $R\supset S'\cup \tau\cup \bigcup_{n\in\omega}\Gamma_n$ such that $R\prec (\Phi'; S')$ and for every countable set $Z\subset R$ there is a countable set $M(Z)\subset R$ satisfying $M(Z)\prec (\Phi';Z)$. Set
\[\M=\{M\in[R]^{\omega}\colon M\prec(\Phi',S')\},\]
ordered by inclusion. The $\sigma$-completeness of $\M$ follows from \cite[Lemma 2.4]{C12}. To see that $\M$ is up-directed, let $N_1,N_2\in \M$, then  the set $M(N_1\cup N_2)\in \M$ and satisfies $N_1\cup N_2\subset M(N_1\cup N_2)$.
By Theorem~\ref{thm:canonicalRetraction}, every $M\in\M$ admits canonical retraction $r_M$ associated to $M$, $K$ and $D$. 
Note that for every $U\in\tau$ there is $M\in\M$ with $U\in M$ (it suffices to put $M=M(\{U\}\cup S')$) and so $\tau\cap (\bigcup \M) = \tau$ is indeed a basis of the topology $\tau$. Therefore it follows from Theorem \ref{thm:canonicalRetraction} \ref{it:weight} and Lemma \ref{l:propertiesOfCanonicalRetractions} that $(r_M)_{M\in\M}$ is a retractional skeleton on $K$. 
Now for every $n\in\omega$, let $\phi_n:\M\to \Gamma_n$ be the mapping defined by $\phi_n(M):=\sup (\Gamma_n\cap M)$. By Theorem \ref{thm:canonicalRetraction} \ref{it:characterizationSecond}, we have that $r_M=r^n_{\phi_n(M)}$, for every $n \in \omega$. Now, let $(M_k)_{k\in\omega}\subset \M$ be an increasing sequence, then it is easy to see that $\sup_k \phi_n(M_k) = \phi_n(M_\infty)$, where $M_{\infty}=\bigcup_{k\in\omega}M_k$.
It remains to prove that the set $\{\phi_n(M):M\in\M\}$ is cofinal in $\Gamma_n$, for each $n\in\omega$. Let $n\in\omega$ and $s\in \Gamma_n$, then there exists $M\in\M$ such that $s\in M$ (it suffices to put $M=M(\{s\}\cup S')$). Therefore $\phi_n(M)=\sup(\Gamma_n\cap M)\geq s$. This concludes the proof.
\end{proof}

\section{Valdivia embedding of compact spaces admitting a commutative skeleton}\label{sec:valdivia}

Giving a compact space $K$ that admits a commutative retractional skeleton $\mathfrak{s}$, it is known that there exists a homeomorphic embedding $h:K\to [-1,1]^I$ such that $h[D(\mathfrak{s})]\subset \Sigma(I)$ (that is, $K$ is Valdivia and $D(\mathfrak{s})$ is a $\Sigma$-subset of $K$). The main aim of this section is to have a very concrete and very flexible way of understanding the mapping $h$, which is the topic handled in Subsection~\ref{subsec:valdiviaEmbedding}, where the proof of Theorem~\ref{thm:Intro3} is given. Apart from Theorem~\ref{thm:Intro3}, we would like to highlight Theorem~\ref{thm:valdiviaThroughModels} which gives a new characterization of Valdivia compacta using suitable models.

\subsection{Canonical retractions associated to suitable models in Valdivia compact spaces} The goal here is to obtain the following concrete description of the canonical retractions associated to suitable models in Valdivia compact spaces.

\begin{lemma}\label{l:canonicalRetractionInSigmaProduct}
For every suitable model $M$ the following holds: Let $D\subset \Sigma(I)$ be such that $K=\overline{D}$ is compact and $\{K,D,I,\tau\}\subset M$ (where $\tau$ is the topology on $K$). Then the mapping $r:K\to K$ defined as $r(x) = x|_{I\cap M}$, for every $x\in K$, is the canonical retraction associated to $M$, $K$ and $D$.
\end{lemma}

This was in a certain sense most probably well-known for countable models (see e.g. \cite[Lemma 2.4]{KM06}), here we show that the situation is the same for uncountable models as well. The remainder of this subsection is more-or-less devoted to the proof of Lemma~\ref{l:canonicalRetractionInSigmaProduct}. We start with two preliminary results.

\begin{lemma}\label{lem:restrictionToIndicesFromModels}
For every suitable model $M$ the following holds: Let $X\subset\Sigma(I)$ be such that $\{X,I\}\subset M$. Then we have
\[\overline{X\cap M} = \overline{\{x|_{I\cap M}\colon x\in X\}}.\]
\end{lemma}
\begin{proof}
Let $S$ and $\Phi$ be the union of sets and lists of formulas from the statements of Lemma~\ref{l:basics-in-M}, Lemma~\ref{l:basics-in-M-2} and Lemma~\ref{l:abstraction}
and let $M\prec (\Phi; S\cup\{X,I\})$. Let $\pi:I\to \er^X$ be the mapping given by $\pi_i(x)=x(i)$, for every $x\in X$ and $i\in I$. By Lemma~\ref{l:basics-in-M} \eqref{it: Dom-Rng-in-M} and Lemma~\ref{l:basics-in-M-2} \eqref{it:Pi-in-M}, we have that $\pi\in M$ and $\A:=\pi[I\cap M]\subset M$. Let $q_M:X\to \er^{\A}$ be the mapping from Lemma~\ref{l:abstraction}, that is, for $x\in X$ we have $q_M(x)(\pi_i) = x(i)$, $i\in I\cap M$. 
Consider the mapping $\phi:\er^{\A}\to \er^I$ given for $x\in \er^{\A}$ by $\phi(x)(i):=x(\pi_i)$, if $i\in I\cap M$ and $\phi(x)(i):=0$, if $i\in I\setminus M$. It is easy to see that $\phi$ is continuous. 
By Lemma~\ref{l:abstraction}, we have $q_M[X]\subset \overline{q_M[X\cap M]}$ which implies that \[\{x|_{I\cap M}\colon x\in X\} = \phi(q_M[X])\subset \overline{\phi(q_M[X\cap M])} = \overline{\{x|_{I\cap M}\colon x\in X\cap M\}}.\] 
Thus, it suffices to note that for every $x\in X\cap M$ we have $x|_{I\cap M} = x$, which follows from the fact that the support of every $x\in X\cap M$ is contained in $M$, see Lemma~\ref{l:basics-in-M-2} \eqref{it:supportSubsetOfM}.
\end{proof}

The following is well-known. We did not find a suitable reference, but it follows e.g. from the proof of \cite[Theorem 6.1]{KM06} (for the key step see also \cite[Lemma 1.2]{AMN88}). For the convenience of the reader we show a short argument based on our previous considerations. 

\begin{lemma}\label{lem:existenceOfTheSkeleton}
Let $K\subset \er^I$ be a compact space and let $D:=\Sigma(I)\cap K$ be dense in $K$. Put
\[\Gamma:=\{A\in [I]^{\leq \omega}\colon x|_A\in K\text{ for every }x\in K\}\]
and for every $A\in\Gamma$ define $r_A:K\to K$ by $r_A(x):=x|_A$, $x\in K$. Then $(r_A)_{A\in\Gamma}$ is a commutative retractional skeleton on $K$ inducing the set $D$.
\end{lemma}
\begin{proof}
It is obvious that each $r_A$ is a continuous retraction with $r_A[K]$ metrizable. For every $A,B\in\Gamma$, we have $A\cap B\in\Gamma$ and $r_A\circ r_B = r_{A\cap B}$ which implies that $r_A\circ r_B = r_B\circ r_A$. Having an increasing sequence $(A_n)_{n\in\omega}$ from $\Gamma$ and $x\in K$, we have $r_{A_n}x\to x|_{\bigcup A_n}$ and so $A_\infty:=\bigcup A_n\in\Gamma$ and $r_{A_n}x\to r_{A_\infty}x$. Let us now observe that for every $x\in D$ there is $A\in\Gamma$ with $r_Ax = x$. Indeed, any $x\in D$ has a countable support, so it suffices to see that for every countable $E\subset I$ there is $A\in\Gamma$ with $E\subset A$. Indeed, by Theorem~\ref{thm:modelExists} and Lemma~\ref{lem:restrictionToIndicesFromModels} (applied to $X=D$), there exists a countable set $M$ such that $E\subset M\cap I$ and $\overline{\{x|_{I\cap M}\colon x\in D\}}=\overline{D\cap M}$, which implies that $M\cap I\in \Gamma$.
Finally, note that the cofinality of $\Gamma$ in $[I]^{\leq \omega}$ implies that the net $(r_A)_{A \in \Gamma}$ converges pointwise to the identity in $K$.
Thus, $\Gamma$ is cofinal and $\sigma$-closed in $[I]^{\leq \omega}$ (in particular $\Gamma$ is a $\sigma$-complete up-directed set) and $(r_A)_{A\in\Gamma}$ is a commutative retractional skeleton on $K$ inducing the the set $D$.
\end{proof}

\begin{proof}[Proof of Lemma~\ref{l:canonicalRetractionInSigmaProduct}]
Let $S$ and $\Phi$ be the union of the countable sets and finite lists of formulas from the statements of Lemma~\ref{l:basics-in-M}, Lemma~\ref{l:basics-in-M-2}, Theorem~\ref{thm:canonicalRetraction} and Lemma~\ref{lem:restrictionToIndicesFromModels}. Pick $M\prec (\Phi; S\cup \{K,D,I,\tau\})$. Since $D$ is dense in $K$ and contained in the set induced by a retractional skeleton (see e.g. Lemma~\ref{lem:existenceOfTheSkeleton}), by Theorem~\ref{thm:canonicalRetraction}, $M$ admits canonical retraction $r_M$ associated to $M$, $K$ and $D$.

By Lemma~\ref{lem:restrictionToIndicesFromModels}, using the continuity of the mapping $K\ni x\mapsto x|_{I\cap M}$ and compactness of $K$, we have $\overline{D\cap M} = \{x|_{I\cap M}\colon x\in K\}$ and so the retraction $r$ is well-defined, continuous and $r[K]=\overline{D\cap M}$.
Let $\pi:I\to {\er^K}$ be the mapping given for $i\in I$ and $x\in K$ as $\pi(i)(x):=x(i)$. By Lemma~\ref{l:basics-in-M-2} \eqref{it:one-to-one-imageOfM} and \eqref{it:Pi-in-M}, we have $\pi\in M$ and $\pi[I]\cap M = \pi[I\cap M]$. Since we obviously have $f\circ r = f$, for every $f\in \pi[I\cap M] = \pi[I]\cap M$ and $\pi[I]\in M$ separates the points of $K$, using Theorem~\ref{thm:canonicalRetraction} we obtain that $r=r_M$.
\end{proof}

\subsection{Valdivia embedding}\label{subsec:valdiviaEmbedding}

This subsection is devoted to the proof of  Theorem~\ref{thm:Intro3} whose proof is based on the proof of \cite[Theorem 2.6]{CGR17}. Quite surprisingly, the inductive argument does not give us only the ``Valdivia embedding'' (that is, Theorem~\ref{thm:Intro3} \ref{it:inducedByCommutative}$\Rightarrow$\ref{it:sigmaEmbedding}), but it also provides us with a new characterization of Valdivia compacta (that is, Theorem~\ref{thm:Intro3} \ref{it:inducedByCommutative}$\Leftrightarrow$\ref{it:skeletonInvariantSubset}) which we use later. Let us also highlight that there is an analogy of this new characterization in the language of suitable models, see Theorem~\ref{thm:valdiviaThroughModels}. We start with a lemma.

\begin{lemma}\label{lem:subsetSeparating}
Let $K$ be a compact space and let $\A\subset \C(K)$ be a set separating the points of $K$. Then there exists $\A'\subset\A$ with $|\A'|=\w(K)$ which separates the points of $K$.
\end{lemma}
\begin{proof}
First, since by the Stone-Weierstrass theorem $\operatorname{alg}(\A \cup \{1\})$ is dense in $\C(K)$, we easily observe that $\{f^{-1}(-1/2,1/2)\colon f\in\operatorname{alg}(\A \cup \{1\})\}$ is a basis for the topology of $K$. Thus, by \cite[Theorem 1.1.15]{Eng}, there is $\F\subset \operatorname{alg}(\A \cup \{1\})$ with $|\F| = \w(K)$ such that $\{f^{-1}(-1/2,1/2)\colon f\in\F\}$ is a basis for the topology of $K$. Pick $\A'\subset \A$ such that $|\A'| = \w(K)$ and $\F\subset \operatorname{alg}(\A' \cup \{1\})$. Then $\A'$ separates the points of $K$, because otherwise $\operatorname{alg}(\A' \cup \{1\})$ would not separate the points of $K$, a contradiction with the fact that $\{f^{-1}(-1/2,1/2)\colon f\in\F\}$ is a basis of the topology.
\end{proof}

\begin{proof}[Proof of Theorem~\ref{thm:Intro3}]
    \ref{it:inducedByCommutative}$\Rightarrow$\ref{it:subsetOfComutativeInducedSubset} Let $\mathfrak{s}_2$ be the commutative retractional skeleton on $K$ inducing $D(\mathfrak{s})$. Then by Theorem~\ref{thm:subskeletons} there is a weak subskeleton of both $\mathfrak{s}$ and $\mathfrak{s_2}$, which easily implies that there is a cofinal subset $\Gamma''\subset \Gamma$ such that $r_s\circ r_t = r_t\circ r_s$, for every $s,t\in\Gamma''$. Thus, it suffices to let $\Gamma'=(\Gamma'')_{\sigma}$. \\
    \ref{it:subsetOfComutativeInducedSubset}$\Rightarrow$\ref{it:skeletonInvariantSubset}: Let $\mathfrak{s}_2=(r_s)_{s\in\Gamma'}$ be a commutative subskeleton of $\mathfrak{s}$. Pick an up-directed set $\Gamma''\subset \Gamma'$ and $x\in D(\mathfrak{s}) = D(\mathfrak{s}_2)$. By Lemma~\ref{l: retraction from up-directed subset} the limit $\lim_{s\in\Gamma''}r_s(x)$ exists. Since there exists $s_0\in\Gamma'$ such that $x = r_{s_0} x$, using the commutativity we obtain
    \[\lim_{s\in\Gamma''}r_s(x) = \lim_{s\in\Gamma''}r_s(r_{s_0}x) = r_{s_0}(\lim_{s\in\Gamma''}r_s(x)) \in D(\mathfrak{s}).\]
    Thus, $\mathfrak{s}_2$ satisfies \ref{it:skeletonInvariantSubset} with $D=D(\mathfrak{s})$.\\
    \ref{it:skeletonInvariantSubset}$\Rightarrow$\ref{it:sigmaEmbedding}: First, we may without loss of generality assume that $\mathfrak{s}_2 = \mathfrak{s}$. We will prove the result by induction on $\kappa:=\w(K)$. We may without loss of generality assume that $\lambda = 1$. If $\kappa=\omega$, then by Lemma~\ref{lem:subsetSeparating}, there exists a countable set $\HH\subset \A$ which separates the points of $K$ and this set does the job.
    So let us assume that the result holds for every compact space of weight strictly smaller than $\kappa$. Proposition \ref{prop:rri} together with Theorem~\ref{thm:modelExists} imply the existence of sets $(M_\alpha)_{\alpha \le \kappa}$ satisfying \ref{it:areModels}-\ref{it:areContinuous} and retractions $(r_\alpha)_{\alpha \le \kappa}$ satisfying \ref{it:comutative}-\ref{it:countableSupport2}. Note that using \ref{it:secondFormula}, we obtain that $r_\alpha[D] \subset D$. For every $\alpha < \kappa$, define $\A_\alpha:=\{f\in\C(r_\alpha[K])\colon f\circ r_\alpha\in\A\}$. It is easy to see that, for every $\alpha < \kappa$, the set $\A_\alpha$ is symmetric, closed, convex and bounded. The fact that $\A_\alpha$ separates the points of $r_\alpha[K]$ follows from \ref{it:compositionWithRa} and \ref{comutaralpha3} implies that $f \circ r_s \in \A_\alpha$, for every $f \in \A_\alpha$ and every $s \in (\Gamma \cap M_\alpha)_\sigma$. For every $\alpha < \kappa$, define $D_\alpha:=D \cap r_\alpha[K]\subset D(\mathfrak s) \cap r_\alpha[K]$. Since $r_{\alpha}[D]$ is dense in $r_{\alpha}[K]$ and $r_{\alpha}[D]\subset  D_{\alpha}$, we have that $D_\alpha$  is dense in $r_\alpha[K]$. 
    Therefore the induction hypothesis and \ref{it:skeletonOnSubset} imply that there are sets $T_\alpha\subset \A_{\alpha}$ such that the mapping $\varphi_\alpha:r_\alpha[K]\to [-1,1]^{T_\alpha}$ given by $\varphi_\alpha(x)(t):=t(x)$, $t\in T_{\alpha}$ and $x \in r_\alpha[K]$ is a homeomorphic embedding and $\varphi_\alpha[D_\alpha]\subset \Sigma(T_\alpha)$, for every $\alpha<\kappa$. We may without loss of generality assume that $T_\alpha\cap T_\beta = \emptyset$ for $\alpha\neq \beta$. Now, we put $T=T_0\cup \bigcup_{\alpha<\kappa} T_{\alpha+1}$ and define $\varphi:K\to [-1,1]^T$ by
    \[
\varphi(x)(t):=\begin{cases}
\tfrac{1}{2}\big(\varphi_{\alpha+1}(r_{\alpha+1}(x))(t) - \varphi_{\alpha+1}(r_{\alpha}(x))(t)\big), & t\in T_{\alpha+1},\\
\varphi_0(r_0(x))(t), & t\in T_0.
\end{cases}
\]
Then $\varphi$ is of course continuous. Let us verify that it is one-to-one. Indeed, if $x, y$ are distinct points from $K$ then by \ref{it:ordinalSeparatingPoints} there is a minimal ordinal $\alpha_0<\kappa$ for which $r_{\alpha_0}(x)\neq r_{\alpha_0}(y)$ and $\alpha_0=0$ or it is a successor ordinal. If $\alpha_0=0$, then there exists $t \in T_0$ such that $\varphi_0\big(r_0(x)\big)(t) \ne \varphi_0\big(r_0(y)\big)(t)$ and so we have $\varphi(x)(t) \ne \varphi(y)(t)$. Otherwise, $\alpha_0=\beta_0+1$ for some $\beta_0<\kappa$ and there is $t\in T_{\alpha_0}$ such that $\varphi_{\alpha_0}(r_{\alpha_0}(x))(t)\neq \varphi_{\alpha_0}(r_{\alpha_0}(y))(t)$. Moreover, since $\alpha_0$ is minimal, we have $r_{\beta_0}(x)=r_{\beta_0}(y)$, hence we obtain $\varphi(x)(t) \neq \varphi(y)(t)$ and so $\varphi(x)\neq \varphi(y)$.
Thus, $\varphi$ is a homeomorphic embedding.

Let us show that $\varphi[D(\mathfrak{s})]\subset \Sigma(T)$. Indeed, by \ref{it:countableSupport2}, for every $x\in D$ the set $\{\alpha<\kappa\colon r_{\alpha+1}(x)\neq r_\alpha(x)\}$ is countable. Moreover, since $r_\alpha[D] \subset D_\alpha$, the induction hypothesis ensures that the supports of $\varphi_0\big(r_0(x)\big)$, $\varphi_{\alpha+1}(r_{\alpha+1}(x))$ and $\varphi_{\alpha+1}(r_{\alpha}(x))$ are countable. Therefore the support of $\varphi(x)$ is countable and we obtain $\varphi[D]\subset \Sigma(T)$. Moreover, since $D$ is dense in $K$, by Lemma~\ref{lem:existenceOfTheSkeleton} there is a commutative retractional skeleton $\mathfrak{s}_2$ on $\varphi[K]$ such that $D(\mathfrak{s}_2) = \varphi[K]\cap \Sigma(T)$. Since $\varphi[D(\mathfrak{s})]\cap D(\mathfrak{s}_2)\supset \varphi[D]$, by \cite[Lemma 3.2]{C14} we have that $\varphi[D(\mathfrak{s})] = D(\mathfrak{s}_2)\subset \Sigma(T)$.

Now, let us show that $\pi_t\circ \varphi\in\A$, for every $t\in T$. Firstly, note that $\pi_t\circ \varphi_\alpha\in \A_\alpha$, for every $t\in T_\alpha$ and every $\alpha < \kappa$. If $t\in T_0$, then we have that $\pi_t\circ \varphi = \pi_t\circ \varphi_0\circ r_0\in\A$. Pick $\alpha<\kappa$ and $t\in T_{\alpha+1}$. Then, similarly as above, $\pi_t\circ \varphi_{\alpha+1}\circ r_{\alpha+1}\in \A$ and therefore using \ref{it:compositionWithRa}, we obtain
\[
\pi_t\circ \varphi = \tfrac{1}{2}\big(\pi_t\circ \varphi_{\alpha+1}\circ r_{\alpha+1} - \pi_t\circ \varphi_{\alpha+1}\circ r_{\alpha+1}\circ r_\alpha\big)\in \A.
\]
Omitting some indices, we may without loss of generality assume that the mapping $T\ni t\mapsto f_t:=\pi_t\circ \varphi\in \A$ is one-to-one and so $\HH:=\{f_t\colon t\in T\}$ does the job.
\\
    \ref{it:sigmaEmbedding}$\Rightarrow$\ref{it:inducedByCommutative}: By Lemma~\ref{lem:existenceOfTheSkeleton} there is a commutative retractional skeleton $\mathfrak{s}_2$ on $\varphi[K]$ such that $D(\mathfrak{s}_2) = \varphi[K]\cap \Sigma(I)$. Since $\varphi[D(\mathfrak{s})]\subset D(\mathfrak{s}_2)$, by \cite[Lemma 3.2]{C14} we have that $\varphi[D(\mathfrak{s})] = D(\mathfrak{s}_2)$ and so the set $D(\mathfrak{s})$ is induced by a commutative retractional skeleton.\\
    \ref{it:skeletonInvariantSubset}$\Rightarrow$\ref{it:inducedByCommutative}: follows from \ref{it:skeletonInvariantSubset}$\Rightarrow$\ref{it:sigmaEmbedding}$\Rightarrow$\ref{it:inducedByCommutative} applied to the set $\A:=B_{\C(K)}$.
    \end{proof}

    The following corollary might be well-known, but let us mention it for future reference.
    
    \begin{cor}\label{cor:fullIsCommutative}
    Let $K$ be a compact space and let $\mathfrak{s}$ be a full retractional skeleton on $K$. Then there exists a commutative subskeleton of $\mathfrak{s}$.
    \end{cor}
    \begin{proof}
    Apply Theorem~\ref{thm:Intro3} \ref{it:skeletonInvariantSubset}$\implies$\ref{it:subsetOfComutativeInducedSubset} to $D:=K$.
    \end{proof}
    
    The proof of Theorem~\ref{thm:Intro3} gives us also the following.
    
    \begin{thm}\label{thm:valdiviaThroughModels}
    Let $K$ be a compact space and let $D$ be a set induced by a retractional skeleton on $K$. Then the following are equivalent.
    \begin{enumerate}[label = (\alph*)]
        \item\label{it:DinducedByCommutative} $D$ is induced by a commutative retractional skeleton.
        \item\label{it:equivUsingSuitableModels} For every suitable model $M$ the following holds:
        \[\forall x\in D\; \exists y\in D\cap \overline{D\cap M}\; \forall f\in\C(K)\cap M:\quad f(x)=f(y).\]
    \end{enumerate}
    \end{thm}
    \begin{proof}
    \ref{it:DinducedByCommutative}$\Rightarrow$\ref{it:equivUsingSuitableModels}: By Theorem~\ref{thm:canonicalRetraction}, there is a finite list of formulas $\Phi$ and a countable set $S$ (depending on the compact space $K$ and the set $D$) such that for any $M\prec(\Phi,S)$, $M$ admits canonical retraction $r_M$ and we have $r_M[D]\subset D$. Then \ref{it:equivUsingSuitableModels} follows from Theorem~\ref{thm:canonicalRetraction} \ref{it:characterizationFirst} applied to $\A:=\C(K)$.\\
    \ref{it:equivUsingSuitableModels}$\Rightarrow$\ref{it:DinducedByCommutative}: Follows from the fact that in the proof of Theorem~\ref{thm:Intro3} \ref{it:skeletonInvariantSubset}$\Rightarrow$\ref{it:sigmaEmbedding} we used condition \ref{it:skeletonInvariantSubset} only to ensure that for a suitable model $M_\alpha$ we have $r_{\alpha}[D]\subset D$, which by Theorem~\ref{thm:canonicalRetraction} \ref{it:characterizationFirst} follows from the condition \ref{it:equivUsingSuitableModels} above.
    \end{proof}
 
\section{Characterization of (semi-)Eberlein compacta}\label{sec:semiEberlein}

Here, we apply the results of the preceding sections and characterize (semi)-Eberlein compacta using the notion of an $\A$-shrinking retractional skeleton.

\begin{defin}\label{def:shrinking}
Let $K$ be a countably compact space. Let $\emptyset\neq \A\subset \C(K)$ be a bounded set. The pseudometric $\rho_{\A}$ on $K$ is given as 
\[
\rho_{\A}(k,l):=\sup_{f\in\A} |f(k)-f(l)|,\quad k,l\in K.
\]
If $(r_s)_{s \in \Gamma}$ is a retractional skeleton on $K$ and $D\subset K$, we say that $(r_s)_{s \in \Gamma}$ is {\it $\A$-shrinking with respect to $D$} if for every $x \in D$ and every increasing sequence $(s_n)_{n \in \omega}$ in $\Gamma$ with $s:=\sup_{n\in \omega}s_{n}$, we have that $\lim_{n\in \omega}\rho_{\A}(r_{s_n}(x),r_s(x))=0$. If $(r_s)_{s\in\Gamma}$ is $\A$-shrinking with respect to $K$, then we just write that $(r_s)_{s\in\Gamma}$ is \emph{$\A$-shrinking}.

Finally, given $\varepsilon> 0$ we say that $(r_s)_{s \in \Gamma}$ is  {\it $(\A,\varepsilon)$-shrinking with respect to $D$} if for every $x \in D$ and every increasing sequence $(s_n)_{n \in \omega}$ in $\Gamma$ with $s:=\sup_{n\in \omega}s_{n}$, we have that $\limsup_{n\in \omega}\rho_{\A}(r_{s_n}(x),r_s(x))\leq \varepsilon$.
\end{defin}

Note that if the nonempty and bounded set $\A$ separates the points of $K$, then $\rho_{\A}$ is a metric on $K$.

The aim of this section is to prove the following result from which Theorem~\ref{thm:Intro1} and Theorem~\ref{thm:Intro2} easily follow.

\begin{thm}\label{t:mainEberleinThroughSkeleton2}
Let $K$ be a compact space and let $D\subset K$ be a dense set. Consider the following conditions.
\begin{enumerate}[label = (\roman*)]
\item\label{it:defSemiEberlein} There exists a homeomorphic embedding $h:K\to [-1,1]^I$ such that $h[D] = c_0(I)\cap h[K]$.
\item\label{it:eqdefSemiEberleinThroughSkeleton} There exist a bounded set $\A\subset \C(K)$ separating the points of $K$ and a retractional skeleton $\mathfrak{s} = (r_s)_{s\in\Gamma}$ on $K$ with $D\subset D(\mathfrak{s})$ such that
\begin{enumerate}[label=(\alph*)]
    \item $\mathfrak{s}$ is $\A$-shrinking with respect to $D$,
    \item\label{it:inSpecialCaseEquivalentToCommutativity2} $\lim_{s\in\Gamma'} r_s(x)\in D$, for every $x\in D$ and every up-directed subset $\Gamma'$ of $\Gamma$.
\end{enumerate}

\item\label{it:eqdefSemiEberleinThroughSkeletonAvailableForContImage} There exist a countable family $\A$ of subsets of $B_{\C(K)}$ and a retractional skeleton $\mathfrak{s} = (r_s)_{s\in\Gamma}$ on $K$ with $D\subset D(\mathfrak{s})$ such that
\begin{enumerate}[label=(\alph*)]
    \item For every $A\in\A$ there exists $\varepsilon_A>0$ such that $\mathfrak{s}$ is $(A,\varepsilon_A)$-shrinking with respect to $D$,
    \item\label{cond:fullInBall3} for every $\varepsilon>0$ we have $B_{\C(K)} = \bigcup\{A\in \A\colon \varepsilon_A < \varepsilon\}$, and
    \item\label{it:inSpecialCaseEquivalentToCommutativity3} $\lim_{s\in\Gamma'} r_s(x)\in D$, for every $x\in D$ and every up-directed subset $\Gamma'$ of $\Gamma$.
\end{enumerate}
\item\label{it:subsetCondition} There exists a homeomorphic embedding $h:K\to [-1,1]^J$ such that $h[D] \subset c_0(J)$.
\end{enumerate}
Then (i)$\Rightarrow$(ii)$\Rightarrow$(iii)$\Rightarrow$(iv).
\end{thm}

Let us first give some comments.

\begin{remark}\label{rem:comparison}
The notion of a shrinking retractional skeleton is inspired by \cite{FM18}, where the definition of a shrinking projectional skeleton was given and WCG Banach spaces were characterized using this notion.

Given a retractional skeleton $(r_s)_{s\in\Gamma}$ on a compact space $K$, it is well known that $(P_s)_{s\in\Gamma}$ given by $P_s(f):=f\circ r_s$, $s\in\Gamma$ is a projectional skeleton on $\C(K)$, see e.g. \cite[Proposition 5.3]{K20}. Moreover, if $\emptyset\neq \A\subset \C(K)$ is a bounded set and $(P_s)_{s\in\Gamma}$ is $\A$-shrinking in the sense of \cite[Definition 16]{FM18}, it is not very difficult to observe that $(r_s)_{s\in\Gamma}$ is $\A$-shrinking in the sense of Definition~\ref{def:shrinking}. It is not clear whether the converse holds as it is (at least formally) a stronger condition. Thus, Theorem~\ref{thm:Intro1} allows us in a certain sense to strengthen implication (ii)$\Rightarrow$(i) from \cite[Theorem 21]{FM18}. Since the other implication is easier, Theorem~\ref{thm:Intro1} may be thought of as a topological counterpart and in a certain sense also strengthening of \cite[Theorem 21]{FM18} in the context of $\C(K)$ spaces.
\end{remark}

Notice that the shrinkingness of a retractional skeleton is not a specific property of one particular skeleton. First, observe that any $\A$-shrinking retractional skeleton is also full, whenever $\A\subset \C(K)$ separates the points of $K$, this is generalized in the following.

    \begin{lemma}\label{lem:shrinkingImpliesFull}
    Let $K$ be a compact space, $\A\subset \C(K)$ be a bounded set separating the point of $K$ and let $\mathfrak{s} = (r_s)_{s\in\Gamma}$ be a retractional skeleton on $K$ which is $\A$-shrinking with respect to a set $D$ with $D\supset D(\mathfrak{s})$. Then $D=D(\mathfrak{s})$ and $\lim_{s\in\Gamma'} r_s(x)\in D$, for every $x\in D$ and every up-directed subset $\Gamma'$ of $\Gamma$.
    \end{lemma}
    \begin{proof}
    Fix $x\in D$ and an up-directed set $\Gamma'\subset\Gamma$. Since $\mathfrak{s}$ is $\A$-shrinking with respect to $D$, it is not very difficult to observe (see e.g. \cite[Proposition 20]{FM18}) that there exists an increasing sequence $(s_n)_{n\in\omega}$ in $\Gamma'$ with $s = \sup_n s_n\in\Gamma$ such that $\rho_\A-\lim_{t\in\Gamma'} r_t(x) = r_s(x)$. Therefore, since the limit $\lim_{t\in\Gamma'} r_t(x)$ exists, we obtain $f(\lim_{t\in\Gamma'} r_t(x)) = f(r_s(x))$, for every $f\in \A$. Since $\A$ separates the points of $K$, we deduce that $\lim_{t\in\Gamma'} r_t(x) = r_s(x)\in D(\mathfrak{s})\subset D$. Finally, for $\Gamma'=\Gamma$ we obtain $x = \lim_{s\in \Gamma} r_s(x)\in D(\mathfrak{s})$ and so $D\subset D(\mathfrak{s})$.
    \end{proof}
    
\begin{lemma}\label{ShrinkingCommutativeSkeleton}
Let $K$ be a compact space and $\A\subset \C(K)$ be a bounded set separating the points of $K$. If there exists an $\A$-shrinking retractional skeleton on $K$, then every full retractional skeleton on $K$ admits a weak subskeleton which is $\A$-shrinking and commutative. 
\end{lemma}
\begin{proof}
By Lemma \ref{lem:shrinkingImpliesFull}, there exists an $\A$-shrinking and full retractional skeleton $(\widetilde{r}_i)_{i\in I}$ on $K$. Moreover, by Corollary~\ref{cor:fullIsCommutative} we may without loss of generality assume that $(\widetilde{r}_i)_{i\in I}$ is commutative.
Now, let $(r_{s})_{s\in\Gamma}$ be a full retractional skeleton on $K$. By Theorem \ref{thm:subskeletons}, there exists a retractional skeleton $(R_i)_{i\in\Lambda}$ which is a weak subskeleton of both $(r_s)_{s\in\Gamma}$ and $(\widetilde{r}_i)_{i\in I}$. It is easy to see that $(R_i)_{i \in \Lambda}$ is commutative and $\A$-shrinking.
\end{proof}

In the remainder of this section we provide the proof of  Theorem~\ref{t:mainEberleinThroughSkeleton2}. Let us start with the proof of the implication \ref{it:defSemiEberlein}$\Rightarrow$\ref{it:eqdefSemiEberleinThroughSkeleton}.

 \begin{lemma}\label{lem:semiEbImpliesShrinkingSkeleton}
Let $K$ be a compact space and $D\subset K$ be a dense subset. If there exists a homeomorphic embedding $h:K\to [-1,1]^I$ such that $h[D] = h[K]\cap c_0(I)$, then there are $
\A\subset B_{\C(K)}$, $\mathfrak{s} = (r_s)_{s\in\Gamma}$ such that \ref{it:eqdefSemiEberleinThroughSkeleton} in Theorem~\ref{t:mainEberleinThroughSkeleton2} holds with $
\A$, $\mathfrak{s} = (r_s)_{s\in\Gamma}$ and we moreover have $f\circ r_s\in \A$ for every $f\in\A$ and $s\in\Gamma$.
\end{lemma}
\begin{proof}
We may without loss of generality assume that $K\subset [-1,1]^I$, $D = K\cap c_0(I)$. Pick the commutative retractional skeleton $(r_A)_{A\in\Gamma}$  from Lemma~\ref{lem:existenceOfTheSkeleton} and put $\SSS:=\{\pi_i|_K\colon i\in I\}\cup \{0\}\subset \C(K)$. Clearly $\SSS$ is bounded and separating. Moreover $D$ is obviously contained in the set $K\cap \Sigma(I)$ (which is the set induced by $(r_A)_{A\in\Gamma}$) and it is easy to observe that if $A\in\Gamma$ and $f\in \mathcal{S}$, then $f\circ r_A\in \SSS$.

Now, let us show that $(r_A)_{A \in \Gamma}$ is $\mathcal{S}$-shrinking with respect to $D$. Pick $x\in D$, an increasing sequence $(A_n)_{n\in\omega}$ in $\Gamma$ and put $A = \sup_n A_n= \bigcup_{n\in\omega} A_n$. Fix $\varepsilon>0$ and let $n_0\in \omega$ be such that $\{i\in A\colon |x(i)|>\varepsilon\}\subset A_{n_0}$. Then for every $n\geq n_0$ we obtain $r_{A_n}x(i) - r_Ax(i) = 1_{A\setminus A_n}\cdot x(i)$, therefore for every $i\in I$ we have $|r_{A_n}x(i) - r_Ax(i)| <\varepsilon$; hence $\sup_{f\in \mathcal{S}} |f(r_{A_n}x) - f(r_{A}x)| < \varepsilon$.

Finally, we note that whenever $\Gamma'\subset \Gamma$ is up-directed and $x\in D$, then $y:=\lim_{A\in\Gamma'} r_Ax$ exists by Lemma~\ref{l: retraction from up-directed subset} and moreover if $i \in \suppt y$, then $y(i)=x(i)$. Therefore, we have that $y \in K \cap c_0(I)=D$.
\end{proof}

The most demanding is the proof of the implication \ref{it:eqdefSemiEberleinThroughSkeletonAvailableForContImage}$\Rightarrow$\ref{it:subsetCondition} in Theorem~\ref{t:mainEberleinThroughSkeleton2}. We start with an easy observation.

\begin{lemma}\label{l: retraction from up-directed subset semi-eberlein caseImproved}
Let $K$ be a compact space, $\A\subset \C(K)$ be a bounded set, $\varepsilon>0$ and $D$ be a subset of $K$. Suppose that $(r_s)_{s\in \Gamma}$ is a retractional skeleton on $K$ that is $(\A,\varepsilon)$-shrinking with respect to $D$. For an up-directed set $\Gamma^{'}\subset \Gamma$ let $R_{\Gamma^{'}}$ be as in Lemma~\ref{l: retraction from up-directed subset}. Then for every $x\in D$ we have the following.
\begin{enumerate}[label=(S\alph*)]
    \item\label{it:shrinkingCanonicalRetraction} If $\Gamma^{'}\subset\Gamma$ is up-directed, then 
    there exists $s_0\in \Gamma'$ such that we have
    \[\rho_\A(R_{\Gamma^{'}}(x),r_s(x)) \leq 7\varepsilon,\qquad s\geq s_0, s\in\Gamma'\] 
    \item\label{it:shrinkingUpDirectedCanonicalRetractions} If $\M\subset \mathcal{P}(\Gamma)$ is such that $\M$ is up-directed and each $M\in\M$ is up-directed, then there exists $M_0\in \M$ such that for every $M\in\M$ with $M\supset M_0$ we have
    \[
    \rho_\A(R_M(x),R_{\bigcup \M}(x)) \leq 14\varepsilon.
    \]
\end{enumerate}
\end{lemma}

\begin{proof}Pick $x\in D$.\\
\ref{it:shrinkingCanonicalRetraction} First, let us observe that there exists $s_0\in\Gamma'$ such that
\begin{equation}\label{eq:cauchyShrinkingness}
\rho_\A(r_s(x),r_{s_0}(x)) < 3\varepsilon,\qquad s\geq s_0,s\in\Gamma'.
\end{equation}
Indeed, if this is not the case we inductively construct an increasing sequence $(s_n)$ in $\Gamma'$ with $\rho_\A(r_{s_n}(x),r_{s_{n+1}}(x))\geq 3\varepsilon$, $n\in\en$ which is in contradiction with $(\A,\varepsilon)$-shrinkingness.

Pick $f\in \A$ and $s_1\geq s_0$, $s_1\in\Gamma'$. Since $\lim_{s\in\Gamma'} r_s(x) = R_{\Gamma'}(x)$, there exists $s_2\geq s_1$, $s_2\in\Gamma'$ with $|f(r_{s_2}(x)) - f(R_{\Gamma'}(x))|<\varepsilon$. Therefore, by \eqref{eq:cauchyShrinkingness} we obtain
\[
|f(r_{s_1}(x)) - f(R_{\Gamma'}(x))|\leq \rho_\A(r_{s_1}(x),r_{s_0}(x)) + \rho_\A(r_{s_0}(x),r_{s_2}(x)) + \varepsilon < 7\varepsilon.
\]
Since $f\in A$ was arbitrary, this proves \ref{it:shrinkingCanonicalRetraction}.\\ 
\ref{it:shrinkingUpDirectedCanonicalRetractions} By \ref{it:shrinkingCanonicalRetraction}, there exists $s_0\in \bigcup \M$ such that $\rho_\A(r_s(x),R_{\bigcup \M}(x))\leq 7\varepsilon$, for every $s\geq s_0$, $s\in \bigcup \M$. Let $M_0\in \M$ be such that $s_0\in M_0$. Then, for every $M\in \M$ with $M\supset M_0$ by \ref{it:shrinkingCanonicalRetraction} there exists $s_M\geq s_0$, $s_M\in M$ with $\rho_\A(r_{s_M}(x),R_M(x))\leq 7\varepsilon$, which implies that
\[\rho_\A(R_M(x),R_{\bigcup \M}(x)) \leq 14\varepsilon.\qedhere\]
\end{proof}

The following proposition together with Theorem~\ref{thm:Intro3} is the core of our argument. The idea to use such a result is related to a characterization of Eberlein compacta by Farmaki \cite[Theorem 2.9]{F87} (see also \cite[Theorem 10]{FGZ}). Note however, that our methods enable us to present a self-contained proof.

\begin{prop}\label{prop:semiEberleinFarmakiImproved}
Let $K \subset [-1,1]^I$ be a compact space and for $I'\subset I$ define
\[\mathcal{S}_{K,I'}=\{\pi_i|_K: i \in I'\}.\]
Suppose that $K$ admits a retractional skeleton $\mathfrak{s}=(r_s)_{s \in \Gamma}$ such that $D(\mathfrak{s})\subset \Sigma(I)$ and let $D \subset D(\mathfrak{s})$. Assume that there is a countable family $\A$ consisting of subsets of $I$ such that
\begin{enumerate}
    \item\label{it:epsilon-shrinking} For every $A \in \A$, there exists $\varepsilon_A>0$ such that $(r_s)_{s \in \Gamma}$ is $(\mathcal{S}_{K,A}, \varepsilon_A)$-shrinking with respect to $D$;
    \item\label{it:covers-ball} For every $\varepsilon> 0$, it holds that $I=\bigcup \{A \in \A: \varepsilon_A< \varepsilon\}$;
    \item $\lim_{s\in\Gamma'}r_s(x)\in D$, for every $x\in D$ and every up-directed subset $\Gamma'$ of $\Gamma$.
\end{enumerate}
Then for every $\varepsilon>0$ there is a decomposition $I = \bigcup_{n=0}^\infty I_n^\varepsilon$ such that \[\forall n\;\forall x\in D:\quad |\{i\in I_n^\varepsilon\colon |x(i)| > \varepsilon\}|<\omega.\]
\end{prop}

\begin{proof}
By \cite[Proposition 19.5]{KKLbook}, we may pick a set $J\subset I$ such that $|J|=\w(K)$ and $\suppt{x}\subset J$, for every $x\in K$. By \cite[Lemma 3.2]{C14}, we have that $D(\mathfrak{s})=\Sigma(I)\cap K$ and hence Lemma~\ref{lem:existenceOfTheSkeleton} ensures that $D(\mathfrak{s})$ is induced by a commutative retractional skeleton.
Therefore it follows from Theorem \ref{thm:Intro3} that we may assume the retractional skeleton $\mathfrak{s}$ is commutative. Now, let us prove the result by induction on the weight of $K$. If $K$ has countable weight, then the set $J$ is countable and we may enumerates it as $J:=(j_n)_{n\ge 1}$. For each $\varepsilon>0$, let $I_{0}^{\varepsilon}= I\setminus J$, and $I_{n}^{\varepsilon}=\{j_n\}$, for every $n \ge 1$. Then $I=\bigcup_{n=0}^{\infty}I_{n}^{\varepsilon}$ and \[\forall n\;\forall x\in K:\quad |\{i\in I_n^\varepsilon\colon |x(i)| > \varepsilon\}|\leq 1.\]
Now suppose that $\w(K)=\kappa \ge \omega_1$ and that the result holds for compact spaces of weight less than $\kappa$. 
Proposition \ref{prop:rri} together with Theorem~\ref{thm:modelExists} imply the existence of sets $(M_\alpha)_{\alpha \le \kappa}$ satisfying \ref{it:areModels}-\ref{it:areContinuous} and retractions $(r_\alpha)_{\alpha \le \kappa}$ satisfying \ref{it:comutative}-\ref{it:countableSupport2}. Note that we can assume that $J \subset \bigcup_{\alpha < \kappa} M_\alpha$, by replacing \ref{it:areIncreasing} by the following (stronger) condition:
\[M_{\alpha+1}\prec (\Phi;\{j_\alpha,\alpha\}\cup M_\alpha), \ \forall \alpha < \kappa,\]
where $J=\{j_\alpha: \alpha < \kappa\}$.
Note that, by Lemma~\ref{l:canonicalRetractionInSigmaProduct}, we may assume that $r_\alpha(x) = x|_{I\cap M_\alpha}$, for every $x\in K$ and $\alpha<\kappa$. 
For each $\alpha < \kappa$, it is easy to see that for every $A\in\A$ the retractional skeleton $(r_s|_{r_\alpha[K]})_{s \in (\Gamma \cap M_\alpha)_\sigma}$ given by \ref{it:skeletonOnSubset} is $(\mathcal S_{r_\alpha[K],A},\varepsilon_A)$-shrinking with respect to the set  $D\cap r_\alpha[K]\subset D(\mathfrak{s})\cap r_{\alpha}[K] \subset \Sigma(I) \cap [-1,1]^{I}$. Moreover, if $\Gamma'\subset (\Gamma \cap M_\alpha)_\sigma$ is up-directed and $x \in D \cap r_\alpha[K]$, then using \ref{comutaralpha3} we conclude that $\lim_{s\in\Gamma'} r_s(x) \in D\cap r_{\alpha}[K]$.
Now fix $\varepsilon>0$ and let $I=\bigcup_{n \ge 1}I^\varepsilon_{n,0}$ be the decomposition given by induction hypothesis applied to $r_0[K]$ (using that by \ref{it:skeletonOnSubset} we may apply the inductive hypothesis to $r_0[K]$), that is, for every $y \in D\cap r_0[K]$ and $n \ge 1$ the set
\[\{i \in I^\varepsilon_{n,0}: \vert y(i) \vert >\varepsilon\}\]
is finite. Fix $\alpha < \kappa$, similarly let $I=\bigcup_{n \ge 1}I^\varepsilon_{n,\alpha+1}$ be the decomposition given by the induction hypothesis applied to $r_{\alpha+1}[K]$, that is, for every $y \in D\cap r_{\alpha+1}[K]$ and $n \ge 1$ the set
\[\{i \in I^\varepsilon_{n,\alpha+1}: \vert y(i) \vert >\varepsilon\}\]
is finite. For $A \in \A$ with $\varepsilon_A< \varepsilon/14$ define $I^\varepsilon_{(0,A)}=I \setminus J$ and for every $n \ge 1$ put
\[I^\varepsilon_{(n,A)}=(A\cap J \cap I^\varepsilon_{n,0} \cap M_0) \cup \bigcup_{\alpha < \kappa} \big(A\cap J \cap I^\varepsilon_{n, \alpha+1} \cap (M_{\alpha+1} \setminus M_{\alpha})\big).\]
Note that $I=\bigcup \{I^\varepsilon_{(n,A)}\colon n \ge 0, A\in \A, 14\varepsilon_A<\varepsilon\}$, since $J \subset \bigcup_{\alpha < \kappa}M_\alpha$ and $I=\bigcup \{A \in \A: 14\varepsilon_A< \varepsilon\}$.
Fixed $x \in D$, $A \in \A$ with $\varepsilon_A< \varepsilon/14$ and $n \ge 0$, let us show that the set
\[S_{(n,A)}=\{i \in I^\varepsilon_{(n,A)}: \vert x(i) \vert >\varepsilon\}\]
is finite. Since $\suppt(x) \subset J$, we have that $S_{(0,A)}$ is empty.
Fixed $n \ge 1$, note that in order to conclude that $S_{(n,A)}$ is finite it suffices to prove that the set \[\Lambda_{(n,A)}=\{\alpha<\kappa\colon |x(i)|>\varepsilon\text{ for some }i\in A\cap J \cap I^\varepsilon_{n,\alpha+1} \cap (M_{\alpha+1} \setminus M_\alpha)\}\]
is finite. Indeed, using that $r_0(x) = x|_{I\cap M_0} = x|_{J\cap M_0}$ we obtain that:
\[S_{(n,A)} \cap (A\cap J \cap I^\varepsilon_{n,0} \cap M_0) \subset \{i \in I^\varepsilon_{n,0}: \vert r_0(x)(i) \vert >\varepsilon\}\]
and therefore, since $r_0(x) = \lim_{s\in(\Gamma\cap M_0)} r_s(x)\in D\cap r_0[K]$, we conclude that $S_{(n,A)} \cap (A\cap J \cap I^\varepsilon_{n,0} \cap M_0)$ is finite. Similarly, for $\alpha < \kappa$ we have
\[\begin{split}
S_{(n,A)} \cap (A\cap J \cap I^\varepsilon_{n,\alpha+1} \cap M_{\alpha+1} \setminus M_\alpha) & \subset \{i \in I^\varepsilon_{n,\alpha+1}: \vert r_{\alpha+1}(x)|_{I \setminus M_\alpha}(i)\vert  >\varepsilon\}\\
& \subset \{i \in I^\varepsilon_{n,\alpha+1}: \vert r_{\alpha+1}(x)(i)\vert  >\varepsilon\}\end{split}\]
and therefore $S_{(n,A)} \cap (A\cap J \cap I^\varepsilon_{n,\alpha+1} \cap M_{\alpha+1} \setminus M_\alpha)$ is finite.
It remains to prove that $\Lambda_{(n,A)}$ is finite. In order to do that suppose by contradiction that $\Lambda_{(n,A)}$ is infinite, so there is a strictly increasing sequence $(\alpha_k)_{k \ge 1}$ of elements of $\kappa$ and a sequence $(i_k)_{k \ge 1}$ such that $i_k\in A\cap J \cap I^\varepsilon_{n,\alpha_k+1} \cap (M_{\alpha_k+1} \setminus M_{\alpha_k})$ and  $|x(i_k)|>\varepsilon$, for every $k \ge 1$.
Put $\alpha = \sup_k \alpha_k$. Then we have (because $i_k\in M_{\alpha_k+1}\setminus M_{\alpha_k}\subset M_\alpha\setminus M_{\alpha_k}$):
 \[\varepsilon < |x(i_k)| = |r_\alpha(x)(i_k) - r_{\alpha_k}(x)(i_k)|\leq \rho_{\mathcal{S}_{(K,A)}} (r_{\alpha}(x),r_{\alpha_k}(x)),\]
for every $k \ge 1$. 
This is a contradiction, because using \ref{it:secondFormula} and Lemma \ref{l: retraction from up-directed subset semi-eberlein caseImproved}  \ref{it:shrinkingUpDirectedCanonicalRetractions} applied to $\M=\{M_{\alpha_k} \cap \Gamma: k \ge 1\}$, we conclude that \[\limsup_{k \to \infty}\rho_{\mathcal{S}_{(K,A)}} (r_{\alpha}(x),r_{\alpha_k}(x))\leq 14\varepsilon_A < \varepsilon,\] since $\mathfrak{s}$ is $(\mathcal{S}_{K,A},\varepsilon_A)$-shrinking with respect to $D$.
\end{proof}

The following is based on \cite[Theorem 10]{FGZ}.

\begin{prop}\label{prop:decompImpliesSemiEberlein}
Let $K \subset [-1,1]^I$ be a compact space and $D$ be a subset of $K$. If for every $\varepsilon>0$, there exists a decomposition $I=\bigcup_{n \in \omega} I^\varepsilon_n$ such that for every $x \in D$ and every $n \in \omega$ the set
\[\{i \in I^\varepsilon_n: \vert x(i) \vert> \varepsilon\}\]
is finite, then there is a homeomorphic embedding $\Phi:K\to [-1,1]^{I\times\omega}$ such that $\Phi[D]\subset c_0(I\times \omega)$.
\end{prop}

\begin{proof}
Let $k\in\omega$ and define the function $\tau_k:\mathbb{R}\to \mathbb{R}$ as
\begin{equation*}
\tau_k(t)=\begin{cases}
t + \frac{1}{k}, & \mbox{if } t\leq -\frac{1}{k},\\
0, & \mbox{if } -\frac{1}{k}\leq t \leq \frac{1}{k},\\
t-\frac{1}{k}, & \mbox{if } t\geq \frac{1}{k}.
\end{cases}
\end{equation*}
Define then $\Phi:K\to [-1,1]^{I \times\omega}$ by
\begin{equation*}
    \Phi(x)(i,k)=\frac{1}{nk}\tau_k(x(i)),
\end{equation*}
if $i\in I^{1/k}_{n}$, $n\in\omega$ and $k\in\omega$. Since the map $\pi_{(i,k)}\circ\Phi:K\to \mathbb{R}$ is continuous, for every $(i,k)\in I\times \omega$, the map $\Phi$ is continuous as well.
The map is also one-to-one. Indeed, for distinct $x_1,x_2 \in K$ there exists an $i \in I$ with $x_1(i)\neq x_2(i)$. Let $k\in\omega$ be such that $1/k<\max\{|x_1(i)|,|x_2(i)|\}$ and pick $n\in\omega$ with $i\in I_n^{1/k}$. Then $\tau_k(x_1(i))\neq\tau_k(x_2(i))$, therefore $\Phi(x_1)(i,k)\neq\Phi(x_2)(i,k)$.

It remains to prove that $\Phi[D]$ is contained in $c_{0}(I\times \omega)$. In order to do that, let $x\in D$ and fix $\varepsilon>0$. If $n,k\in\omega$, and $n>1/\varepsilon$ or $k>1/\varepsilon$, then $|\Phi(x)(i,k)|<\varepsilon$, for any choice of $i\in I^{1/k}_n$. Let $n,k<1/\varepsilon$ (observe that there are only finitely many $n$ and $k$ such that this inequality holds). Then we have
\begin{equation*}
    \begin{split}
        \{i\in I_{n}^{1/k}:|\Phi(x)(i,k)| & >\varepsilon\}  \subseteq  \{i\in I_{n}^{1/k}:\tau_k(x(i))\neq 0\}\\
        & =\{i\in I_{n}^{1/k}:x(i)>1/k\}\cup \{i\in I^{1/k}_n: x(i)<-1/k\}\\
        & =\{i \in I_{n}^{1/k}: \vert x(i) \vert>1/k\}.
    \end{split}
\end{equation*}
Therefore, the set $\{i\in I_{n}^{1/k}:|\Phi(x)(i,k)|>\varepsilon\}$ is finite and thus we conclude that $\Phi(x)\in c_0(I\times \omega)$.
\end{proof}

\begin{proof}[Proof of Theorem~\ref{t:mainEberleinThroughSkeleton2}]
Lemma~\ref{lem:semiEbImpliesShrinkingSkeleton} ensures that $\ref{it:defSemiEberlein} \Rightarrow \ref{it:eqdefSemiEberleinThroughSkeleton}$.

Now let us prove that \ref{it:eqdefSemiEberleinThroughSkeleton}$\Rightarrow$\ref{it:eqdefSemiEberleinThroughSkeletonAvailableForContImage}. Let $\A$ and $\mathfrak{s}=(r_s)_{s\in\Gamma}$ be as in the assumption and let $\lambda \ge 1$ be such that $\A \subset \lambda B_{\C(K)}$. We may without loss of generality assume that the constant $1$ function is  member of $\A$. For every $n\in\en$ put
\[
\A_n:=\left\{\sum_{i=1}^k a_i\Pi_{j=1}^n f_{i,j}\colon f_{i,j}\in\A,\; k\in\en\;, \sum_{i=1}^k |a_i|\leq n\right\}
\]
and for $m\in\en$ we further put
\[
\A_{n,m}:=(\A_n + \tfrac{1}{2m}B_{\C(K)})\cap B_{\C(K)}.
\]
Now, we claim that the family $\widetilde{\A}:=\{\A_{n,m}\colon n,m\in\en\}$ and the retractional skeleton $\mathfrak{s}$ satisfy the condition from \ref{it:eqdefSemiEberleinThroughSkeletonAvailableForContImage}. Pick $n,m\in\en$. Then $\mathfrak{s}$ is $(\A_{n,m},\tfrac{1}{m})$-shrinking with respect to $D$. Indeed, given $x\in D$ and an increasing sequence $(s_k)$ in $\Gamma$ with $s = \sup s_k$, we have
\[
\rho_{\A_{n,m}}(r_{s_k}(x),r_s(x))\leq \rho_{\A_{n}}(r_{s_k}(x),r_s(x)) + \tfrac{1}{m}\leq n^2\lambda^{n+1}\rho_{\A}(r_{s_k}(x),r_s(x)) + \tfrac{1}{m},
\]
so using that $\mathfrak{s}$ is $\A$-shrinking with respect to $D$, we obtain \[\limsup_k \rho_{\A_{n,m}}(r_{s_k}(x),r_s(x))\leq \tfrac{1}{m}.\]
Finally, since $\bigcup_{n\in\en} \A_n = \alg(\A)$ is norm-dense in $\C(K)$ we easily observe that $B_{\C(K)} = \bigcup_{n\in\en} \A_{n,m}$ for every $m\in\en$ from which the condition \ref{cond:fullInBall3} follows.

Now let us prove that \ref{it:eqdefSemiEberleinThroughSkeletonAvailableForContImage}$\Rightarrow$\ref{it:subsetCondition}. Let $\A$ and $\mathfrak{s}=(r_s)_{s\in\Gamma}$ be as in the assumption. By Theorem~\ref{thm:Intro3}, there exists $\HH\subset B_{\C(K)}$ such that the mapping $\varphi:K\to [-1,1]^\HH$ given by $\varphi(x)(h):=h(x)$, for $h\in\HH$ and $x\in K$, is a homeomorphic embedding and $\varphi[D(\mathfrak s)]\subset \Sigma(\HH)$.
For every $s \in \Gamma$, define $q_s=\varphi \circ r_s \circ \varphi^{-1}:\varphi[K] \to \varphi[K]$ and note that the retractional skeleton $(q_s)_{s \in \Gamma}$ is $(\mathcal{S}_{(\varphi[K],\HH\cap A)},\varepsilon_A)$-shrinking with respect to the set $\varphi[D]\subset\varphi[D(\mathfrak s)]$ for every $A\in\A$. Indeed, fix $x \in D$ and an increasing sequence $(s_n)_{n \ge 1}$ of elements of $\Gamma$ with $s=\sup_n s_n$. Then we have
\[ \rho_{\mathcal{S}_{(\varphi[K],\HH\cap A)}}\Big(q_{s_n}\big(\varphi(x)\big), q_{s}\big(\varphi(x)\big)\Big)=\sup_{h \in \HH\cap A} \Big\vert \varphi \big(r_{s_n}(x)\big)(h)-\varphi \big(r_{s}(x)\big)(h)\Big \vert=\]
\[= \sup_{h \in \HH\cap A} \Big\vert h\big(r_{s_n}(x)\big)-h\big(r_{s}(x)\big)\Big \vert \le  \rho_A(r_{s_n}(x),r_{s}(x)),\]
so since $\mathfrak{s}$ is $(A,\varepsilon_A)$-shrinking with respect to the set $D$, $(q_s)$ is $(\mathcal{S}_{(\varphi[K],\HH\cap A)},\varepsilon_A)$-shrinking with respect to the set $\varphi[D]$. Obviously, we have $\HH = \bigcup\{A\cap\HH\colon \varepsilon_A<\varepsilon\}$ for every $\varepsilon>0$. Finally, for every up-directed set $\Gamma'\subset \Gamma$ and every $x\in D$ we have that $\lim_{s\in\Gamma'} q_s(\varphi(x)) = \varphi (\lim_{s\in\Gamma'} r_s(x)) \in \varphi[D]$. Therefore, the result follows from Proposition~\ref{prop:semiEberleinFarmakiImproved} and Proposition \ref{prop:decompImpliesSemiEberlein}.
\end{proof}

\section{Applications to the structure of (semi)-Eberlein compacta}\label{sec:contImages}

We collect our applications to the structure of (semi-)Eberlein compacta. Most importantly, we prove Theorem~\ref{thm:Intro4}.

\subsection{Eberlein compacta}

As mentioned above, using Theorem~\ref{thm:Intro1} it is not very difficult to show that any continuous image of Eberlein compacta is Eberlein. The reason is that for continuous images of Eberlein compacta it is quite standard to verify the condition \ref{it:eqdefSemiEberleinThroughSkeletonAvailableForContImage} from Theorem~\ref{t:mainEberleinThroughSkeleton2}. We will not provide here the full argument as it is possible to further generalize this observation, see Remark~\ref{rem:contImageEberlein} below. The remainder of this subsection is devoted to the proof of Theorem~\ref{thm:countableCompactEberlein}, which is a generalization of Theorem~\ref{thm:Intro1}, where instead of compactness we assume countable compactness. In order to show the argument, we need a lemma first.
Recall that every real-valued continuous function defined on a countably compact space $D$ is bounded so we may consider the supremum norm on $\C(D)$.
\begin{lemma}\label{l:separatePoints}
Let $D$ be a countably compact space. Suppose that there exist a bounded set $\A\subset \C(D)$ separating the points of $D$ and a full retractional skeleton $\mathfrak{s} = (r_s)_{s\in\Gamma}$ on $D$ such that $f\circ r_s\in\A$, for every $f\in\A$ and $s\in\Gamma$. Then $\A'=\{\beta f\colon f\in \A\}$ separates the points of $\beta D$.
\end{lemma}
\begin{proof}
By \cite[Proposition 4.5]{CK15}, there exists a retractional skeleton $\mathfrak{S}=(R_s)_{s\in\Gamma}$ on $\beta D$ such that $D(\mathfrak{S}) = D$ and $R_s|_D = r_s$, for every $s\in\Gamma$. Let $x,y \in \beta D$ be distinct points. Since $\lim_{s\in\Gamma}R_s(x)=x$ and $\lim_{s\in\Gamma}R_s(y)=y$, there exists $s\in \Gamma$ such that $R_s(x)\neq R_s (y)$. Since $R_s(x), R_s(y)\in D$, there exists a function $f\in \A$ such that $f(R_s(x))\neq f(R_s(y))$. Therefore we have $\beta f(R_s(x))\neq\beta f(R_s(y))$. It is easy to see that $(\beta f \circ R_s)|_D= f \circ r_s$, which implies that $\beta f \circ R_s= \beta(f \circ r_s) \in \A'$, since $f \circ r_s\in\A$.
\end{proof}

\begin{remark}
Note that the assumption ``$f\circ r_s\in \A$, for every $f\in \A$ and $s\in\Gamma$'' in Lemma~\ref{l:separatePoints} is essential. Indeed, consider $D=[0,\omega_1)$ and \[\A = \{1_{\{0\}\cup [\alpha+1,\omega_1)}\colon \alpha < \omega_1\}.\]
Then it is easy to see that $\A$ separates the points of $D$ and that $D$ admits the full retractional skeleton $(r_\alpha)_{\alpha<\omega_1}$ given by the formula
\begin{equation*}
   r_{\alpha}(\beta)= \begin{cases} \beta & \beta\leq\alpha \\ \alpha+1 & \alpha<\beta<\omega_1,\end{cases}
\end{equation*}
for every $\alpha<\omega_1$.
However, we have $\beta D = [0,\omega_1]$ and the set \[\A' = \{\beta 1_{\{0\}\cup [\alpha+1,\omega_1)}\colon \alpha<\omega_1\} = \{1_{\{0\}\cup [\alpha+1,\omega_1]}\colon \alpha<\omega_1\}\]
does not separate $0$ from $\omega_1$.
\end{remark}

\begin{thm}\label{thm:countableCompactEberlein}
Let $D$ be a countably compact space. 
Then the following conditions are equivalent.
\begin{enumerate}[label = (\roman*)]
    \item There exists a set $I$ such that $D$ embeds homeomorphically into $(c_0(I),w)$.
    \item $D$ is an Eberlein compact space.
    \item There exist a bounded set $\A\subset \C(D)$ separating the points of $D$ and a full retractional skeleton $\mathfrak{s} = (r_s)_{s\in\Gamma}$ on $D$ such that
    \begin{enumerate}[label=(\alph*)]
        \item $\mathfrak{s}$ is $\A$-shrinking, 
        \item\label{it: compositionassumption} $f\circ r_s\in\A$, for every $f\in\A$ and $s\in\Gamma$.
    \end{enumerate}
   \end{enumerate}
\end{thm}
\begin{proof}
$(i)\Rightarrow(ii)$ follows from the classical Eberlein-\v{S}mulian theorem and $(ii)\Rightarrow(iii)$ follows from Theorem~\ref{t:mainEberleinThroughSkeleton2} and Lemma \ref{lem:shrinkingImpliesFull}. If (iii) holds, pick the corresponding set $\A$ and the full retractional skeleton $(r_s)_{s\in\Gamma}$ on $D$. By \cite[Proposition 4.5]{CK15}, there exists a retractional skeleton $\mathfrak{S}=(R_s)_{s\in\Gamma}$ on $\beta D$ such that $D(\mathfrak{S}) = D$ and $R_s|_D = r_s$, for every $s\in\Gamma$.
Consider now the set $\A':=\{\beta f\colon f\in \A\}\subset \C(\beta D)$ and the retractional skeleton $\mathfrak{S}$. By Lemma~\ref{l:separatePoints}, $\A'$ separates the points of $\beta D$. Obviously, $\mathfrak{S}$ is $\A'$-shrinking with respect to $D$ and it is easy to see that for every $f\in \A$ and $s\in\Gamma$ we have $\beta f\circ R_s = \beta(f\circ r_s) \in \A'$. By Lemma \ref{lem:shrinkingImpliesFull} we have that $\lim_{s\in \Gamma'}R_s(x)\in D$, for every $x\in D$ and every up-directed subset $\Gamma'$ of $\Gamma$. Therefore, Theorem~\ref{t:mainEberleinThroughSkeleton2} ensures that (i) holds.
\end{proof}

\subsection{Semi-Eberlein compacta}
 In this subsection we provide new stability results for the class of semi-Eberlein compacta. The most important in this respect is probably Corollary~\ref{cor:images} which implies Theorem~\ref{thm:Intro4}.
 
\begin{lemma}\label{l:contImageConnectionOfSkeletons}
For every suitable model $M$ the following holds: Let $(K,\tau)$ and $(L,\tau')$ be compact spaces, $D\subset K$ a dense subset that is contained in the set induced by a retractional skeleton, and $\varphi:K\to L$ a continuous map such that $\varphi[D]\subset L$ a dense subset that is contained in the set induced by a retractional skeleton. If $\{K,L,\tau,\tau',D,\varphi\}\subset M$, then there are canonical retractions $r_M$ and $R_M$ associated to $M$, $K$ and $D$ and to $M$, $L$ and $\varphi[D]$, respectively, and we have $R_M\circ \varphi = \varphi\circ r_M$.
\end{lemma}

\begin{proof}Let $S$ and $\Phi$ be the union of countable sets and finite lists of formulas from the statements of Lemma~\ref{l:basics-in-M-2} and Theorem~\ref{thm:canonicalRetraction}. Let $M\prec (\Phi; S\cup\{K,L,\tau,\tau',D,\varphi\})$.

The existence of $r_M$ and $R_M$ follows from Theorem~\ref{thm:canonicalRetraction}. Moreover, given $x\in K$ and $f\in \C(L)\cap M$, by Lemma~\ref{l:basics-in-M-2} we have $f\circ \varphi\in \C(K)\cap M$ and so by the definition of $r_M$, $f\circ \varphi\circ r_M = f\circ \varphi$ which implies that $y=\varphi(r_M(x))\in\varphi[\overline{D\cap M}]\subset \overline{\varphi[D]\cap M}$ is a point satisfying $f(y) = f(\varphi(x))$ for every $f\in \C(L)\cap M$ and so by the uniqueness property of $R_M(\varphi(x))$ (see Theorem~\ref{thm:canonicalRetraction} \ref{it:characterizationFirst}) we obtain that $R_M(\varphi(x)) = \varphi(r_M(x))$.
\end{proof}

\begin{thm}\label{t:mainContImage}
Let $K$ be a compact space and $D\subset K$ be a dense subset such that there exists a homeomorphic embedding $h:K\to [-1,1]^J$ such that $h[D] = c_0(J)\cap h[K]$. Let us suppose that $\varphi:K\to L$ is a continuous surjection and $\varphi[D]$ is subset of the set induced by a retractional skeleton on $L$.

Then there is a homeomorphic embedding $H:L\to[-1,1]^I$ with $H[\varphi[D]]\subset c_0(I)$. In particular, $L$ is semi-Eberlein.
\end{thm}
\begin{proof}
By Lemma~\ref{lem:semiEbImpliesShrinkingSkeleton}, there exists a set $\A\subset  B_{\C(K)}$ separating the points of $K$ and a retractional skeleton $\mathfrak{s} = (r_s)_{s\in\Gamma}$ on $K$ with $D\subset D(\mathfrak{s})$ such that $\mathfrak{s}$ is $\A$-shrinking with respect to $D$ and $\lim_{s\in\Gamma'}r_s(x)\in D$ for every $x\in D$ and every up-directed subset $\Gamma'$ of $\Gamma$. Using Lemma~\ref{l:contImageConnectionOfSkeletons} and an argument similar to the one presented in the proof of Theorem~\ref{thm:subskeletons}, we conclude that there are countable set $S$, finite list of formulas $\Phi$ and a set $R$ such that
\[
\M = \{M\in[R]^\omega\colon M\prec(\Phi,S)\}
\]
ordered by inclusion is an up-directed and $\sigma$-complete set and moreover we have
\begin{itemize}
    \item every $M\in\M$ admits canonical retractions $r_M$ and $R_M$ associated to $M$, $K$ and $D$ and to $M$, $L$ and $\varphi[D]$, respectively;
    \item $\mathfrak{s}_K:=(r_M)_{M\in\M}$ is a weak subskeleton of $\mathfrak{s}$ and $\mathfrak{s}_L:=(R_M)_{M\in\M}$ is a retractional skeleton on $L$;
    \item for every $M\in\M$ we have $R_M\circ \varphi = \varphi\circ r_M$.
\end{itemize}
In particular, we have that $D(\mathfrak{s}) = D(\mathfrak{s}_K)$, $\mathfrak{s}_K$ is $\A$-shrinking with respect to $D$ and $\lim_{M\in\M'}r_M(x)\in D$ for every $x\in D$ and every up-directed subset $\M'$ of $\M$. We obviously have $\varphi[D]\subset D(\mathfrak{s}_L)$. Consider now the isometric embedding $\varphi^*:\C(L)\to\C(K)$ given by the formula $\varphi^*f:=f\circ \varphi$, $f\in\C(L)$. Further, similarly as in the proof of  \ref{it:eqdefSemiEberleinThroughSkeleton}$\Rightarrow$\ref{it:eqdefSemiEberleinThroughSkeletonAvailableForContImage} of Theorem~\ref{t:mainEberleinThroughSkeleton2} for every $n\in\en$ put
\[
\A_n:=\left\{\sum_{i=1}^k a_i\Pi_{j=1}^n f_{i,j}\colon f_{i,j}\in\A,\; k\in\en\;, \sum_{i=1}^k |a_i|\leq n\right\},
\]
for $m\in\en$ we further put
\[
\A_{n,m}:=(\A_n + \tfrac{1}{2m}B_{\C(K)})\cap B_{\varphi^*\C(L)},\quad \B_{n,m}:=(\varphi^*)^{-1}(\A_{n,m})
\]
and we observe that $B_{\C(L)} = \bigcup_{n\in\en} \B_{n,m}$ for every $m\in\en$. Moreover, observe that for every $n,m\in\en$ the retractional skeleton $\mathfrak{s}_L$ is $\big(\B_{n,m},\tfrac{1}{m}\big)$-shrinking with respect to $\varphi[D]$. Indeed, given $x\in D$ and an increasing sequence $(M_k)_{k \in \en}$ in $\M$ with $M = \sup_{k \in \en} M_k$, we have
\[\begin{split}
\rho_{\B_{n,m}} & (R_{M_k}(\varphi(x)),R_M(\varphi(x))) = \sup_{f\in\C(L), f\circ \varphi\in\A_{n,m}}\big|f(\varphi(r_{M_k}(x))) - f(\varphi(r_{M}(x)))\big|\\ & \leq \rho_{\A_{n,m}}(r_{M_k}(x),r_M(x)) \leq \rho_{\A_{n}}(r_{M_k}(x),r_M(x)) + \tfrac{1}{m}\\  & \leq n^2\rho_{\A}(r_{M_k}(x),r_M(x)) + \tfrac{1}{m},
\end{split}\]
so using that $\mathfrak{s}$ is $\A$-shrinking with respect to $D$, we obtain \[\limsup_k \rho_{\B_{n,m}}(R_{M_k}(x),R_M(x))\leq \tfrac{1}{m}.\]
Finally, for every $x\in D$ and up-directed subset $\M'$ of $\M$ we have \[\lim_{M\in\M'}R_M(\varphi(x)) = \varphi(\lim_{M\in\M'} r_M(x))\in \varphi[D].\]
Hence, application of  Theorem~\ref{t:mainEberleinThroughSkeleton2} \ref{it:eqdefSemiEberleinThroughSkeletonAvailableForContImage}$\implies$\ref{it:subsetCondition} finishes the proof.
\end{proof}

\begin{cor}\label{cor:firstContImage}
Let $K$ be a semi-Eberlein compact space, $\varphi:K\to L$ be a continuous surjection and $D\subset K$ be a dense subset such that there exists a homeomorphic embedding $h:K\to [-1,1]^J$ with $h[D]= c_0(J)\cap h[K]$. Then $S:=h^{-1}[\Sigma(J)]$ is the unique set induced by a retractional skeleton in $K$ with $D\subset S$. Assume that one of the following conditions holds:
\begin{enumerate}
    \item\label{it:firstSufficientCondition} $\varphi^*\C(L) = \{f\circ \varphi\colon f\in\C(L)\}$ is $\tau_p(S)$-closed in $\C(K)$;
    \item\label{it:secondSufficientCondition} The set $\{(x,y)\in S\times S\colon \varphi(x)=\varphi(y)\}$ is dense in $\{(x,y)\in K\times K\colon \varphi(x)=\varphi(y)\}$.
\end{enumerate}
Then there is a homeomorphic embedding $H:L\to[-1,1]^I$ with $H[\varphi[D]]\subset c_0(I)$. In particular, $L$ is semi-Eberlein.
\end{cor}
\begin{proof}
By Lemma~\ref{lem:existenceOfTheSkeleton}, $\Sigma(J)\cap h[K]$ is induced by a retractional skeleton and so its preimage $S$ is induced by a retractional skeleton as well. The uniqueness of $S$ follows from \cite[Lemma 3.2]{C14}. If \eqref{it:firstSufficientCondition} holds, then by \cite[Theorem 4.5]{cuthSimul} the set $\varphi[S]$ is induced by a retractional skeleton and so we may apply Theorem~\ref{t:mainContImage}. Finally, by \cite[Lemma 2.8]{K00} condition \eqref{it:secondSufficientCondition} implies \eqref{it:firstSufficientCondition}.
\end{proof}

\begin{remark}\label{rem:contImageEberlein}
Note that as a very particular case we obtain that a continuous image of an Eberlein compact space is Eberlein, since in this case we have $D=K$ and thus condition \eqref{it:secondSufficientCondition} in Corollary \ref{cor:firstContImage} is trivially satisfied.
\end{remark}

The following answers the second part of \cite[Question 6.6]{KL04}.

\begin{cor}\label{cor:images}
Let $K$ be a semi-Eberlein compact space, $\varphi:K\to L$ be a continuous surjection. Assume that one of the following conditions holds:
\begin{enumerate}
    \item $K$ is Corson.
    \item $\varphi$ is open and $K$ or $L$ has a dense set of $G_\delta$ points.
\end{enumerate}
Then $L$ is semi-Eberlein.
\end{cor}
\begin{proof}
Let $D\subset K$ be a dense subset such that there exists a homeomorphic embedding $h:K\to [-1,1]^J$ with $h[D]= c_0(J)\cap h[K]$ and put $S:=h^{-1}[\Sigma(J)]$. Note that $S$ is a dense $\Sigma$-subset of $K$ and, by Corollary~\ref{cor:firstContImage}, it is the unique set induced by a retractional skeleton with $D\subset S$.

If $K$ is Corson then it admits a full retractional skeleton, so by the uniqueness of $S$ we have that $S=K$ and thus condition \eqref{it:secondSufficientCondition} from Corollary~\ref{cor:firstContImage} is obviously satisfied.

If $\varphi$ is open and $L$ has dense set of $G_\delta$ points, then by \cite[Lemma 6.1]{cuthSimul} the set $\varphi[S]$ is induced by a retractional skeleton and thus we may apply directly Theorem~\ref{t:mainContImage} (note that inspecting the proof of \cite[Lemma 6.1]{cuthSimul} one can observe that even the condition \eqref{it:secondSufficientCondition} from Corollary~\ref{cor:firstContImage} is satisfied).

Finally, if $\varphi$ is open and $K$ has dense set of $G_\delta$ points then it is easy to see that $L$ has dense set of $G_\delta$ points and we may apply the above.
\end{proof}

Let us note that if $D$ is a dense $\Sigma$-subset of $K$ and $\varphi:K\to L$ is a continuous retract, it may happen that $\varphi[D]$ is not $\Sigma$-subset of $K$, see \cite[Remark 3.25]{K00}. Thus, it is not possible to apply directly Theorem~\ref{t:mainContImage} for the case when $\varphi$ is a retraction and this is basically the reason why we do not know how to answer also the first part of \cite[Question 6.6]{KL04} using our methods.

\section{Open questions and remarks}\label{sec:questions}
    
In Section~\ref{sec:models} we obtained as an application of our methods that for a countable family of retractional skeletons inducing the same set there is a common weak subskeleton, see Theorem \ref{thm:subskeletons}. It would be interesting to know whether we can find even a subskeleton (not only a weak one).

\begin{question}\label{q: weaksubtosub}
In Theorem \ref{thm:subskeletons}, is it possible to obtain a subskeleton instead of a weak subskeleton?
\end{question}

When working with retractional skeletons, their index sets are quite mysterious. For Banach spaces with a projectional skeleton, the index set may be chosen to consist of the ranges of the involved projections (ordered by inclusion), see \cite[Theorem 4.1]{CK14}. We wonder whether something similiar holds for spaces with a retractional skeleton.

\begin{question}
Let $K$ be a compact space and let $D\subset K$ be induced by a retractional skeleton on $K$. Does there exist a family of retractions $\{r_F\colon F\in\F\}$ indexed by a family of compact spaces $\F$ ordered by inclusion satisfying the following conditions?
    \begin{enumerate}[label = (\roman*)]
        \item\label{it:sigma} whenever $(F_n)_{n\in\omega}$ is an increasing sequence from $\F$, then $\sup_n F_n = \overline{\bigcup_n F_n}\in\F$,
        \item\label{it:range} for every $F\in\F$ we have $r_F[K] = F$,
        \item\label{it:sub} $(r_F)_{F\in\F}$ is a retractional skeleton on $K$ inducing the set $D$.
    \end{enumerate}
\end{question}

Note that in Proposition~\ref{prop:semiEberleinFarmakiImproved} we proved a result in a sense very similar to the characterization of Eberlein compacta from \cite[Theorem 2.9]{F87} (see also \cite[Theorem 10]{FGZ}). We wonder whether an analogoue of \cite[Theorem 10]{FGZ} holds also in the context semi-Eberlein compacta. Note that one implication follows from  Proposition~\ref{prop:decompImpliesSemiEberlein}, so a positive answer to the following question would give a characterization of semi-Eberlein compact subspaces of $[-1,1]^I$.

\begin{question}
Let $K\subset [-1,1]^I$ be a compact space such that $\Sigma(I)\cap K$ is dense in $K$. Let $K$ be semi-Eberlein. Does there exist $D\subset \Sigma(I)\cap K$ which is dense in $K$ such that for every $\varepsilon>0$ there exists a decomposition $I = \bigcup_{n=0}^\infty I_n^\varepsilon$ satisfying
\[\forall n\in\omega\; \forall x\in D:\quad |\{i\in I_n^\varepsilon\colon |x(i)|>\varepsilon\}|<\omega?\]
\end{question}

Finally, let us emphasize that we believe that Theorem~\ref{thm:Intro3} gives quite a big flexibility to consider other subclasses of Valdivia compact spaces and characterize them using the notion of retractional skeletons. The reason why we believe so is, that by Theorem~\ref{thm:Intro3} we may consider any set induced by a retractional skeleton to be a subset of $\Sigma(I)$ (where $\A\subset \C(K)$ plays the role of the set $I$); moreover, for subsets of $\Sigma(I)$ several classes of compact spaces were characterized using their evaluations on the set $I$, see \cite{FGZ}. Thus, there is enough room for further possible research by considering those classes of compacta and try to develop the right notion which would give a characterization using retractional skeletons.

\end{document}